\documentclass[letterpaper,10pt,reqno,onefignum,onetabnum]{amsart}
\usepackage[english]{babel}
\usepackage{amsmath}
\usepackage{amsthm}
\usepackage{verbatim}
\usepackage{mathrsfs}
\usepackage{bm}
\usepackage{cite}
\usepackage[foot]{amsaddr}
\usepackage[dvipsnames]{xcolor}
\usepackage{hyperref}
\hypersetup{
	colorlinks = true,
	linkcolor = OliveGreen,
	anchorcolor = OliveGreen,
	citecolor = OliveGreen,
	filecolor = OliveGreen,
	urlcolor = OliveGreen
}
\usepackage{algorithm}
\usepackage{algorithmic}
\usepackage{float}
\usepackage{lipsum}
\usepackage{amsfonts}
\usepackage{amssymb}
\usepackage{graphicx}
\usepackage{hyperref}
\usepackage{textcomp}
\usepackage{epstopdf}
\usepackage{cases}
\usepackage{multirow}
\ifpdf
\DeclareGraphicsExtensions{.eps,.pdf,.png,.jpg}
\else
\DeclareGraphicsExtensions{.eps}
\fi
\usepackage{verbatim}
\usepackage{mathrsfs}
\usepackage{bm}

\newtheorem{teo}{Theorem}[section]
\newtheorem{prop}[teo]{Proposition}

\newtheorem{pro}[teo]{Problem}
\newtheorem{algo}[teo]{Algorithm}
\newtheorem{asume}[teo]{Assumption}

\newtheorem{rem}[teo]{Remark}

\usepackage{lipsum}
\usepackage{amsfonts}
\usepackage{amssymb}
\usepackage{graphicx}
\usepackage{epstopdf}

\usepackage{cases}
\usepackage{multirow}
\usepackage{algorithmic}
\usepackage{tikz}
\usepackage{booktabs}%
\usetikzlibrary{matrix}
\usetikzlibrary{arrows}
\ifpdf
\DeclareGraphicsExtensions{.eps,.pdf,.png,.jpg}
\else
\DeclareGraphicsExtensions{.eps}
\fi
\usepackage{verbatim}
\usepackage{mathrsfs}
\usepackage{bm}
\usepackage{color}
\usepackage{caption}
\usepackage{subcaption}
\usepackage{enumitem}

\newcommand{\N}{\mathbb N}

\newcommand{\R}{\mathbb R}

\renewcommand{\H}{\mathcal{H}}
\newcommand{\G}{\mathcal G}

\newcommand{\HH}{{\bm{\mathcal{H}}}}

\newcommand{\id}{\textnormal{Id}}

\newcommand{\dom}{\textnormal{dom}\,}

\newcommand{\zer}{\textnormal{zer}}

\newcommand{\gra}{\textnormal{gra}\,}

\newcommand{\scal}[2]{{\left\langle{{#1}\mid{#2}}\right\rangle}}

\newcommand{\menge}[2]{\big\{{#1}~\big |~{#2}\big\}}

\newcommand{\RPP}{\ensuremath{\left]0,+\infty\right[}}

\newcommand{\RX}{\ensuremath{\left]-\infty,+\infty\right]}}

\newcommand{\sri}{\ensuremath{\text{\rm sri}\,}}

\newcommand{\prox}{\ensuremath{\text{\rm prox}\,}}

\newcommand{\weakly}{\ensuremath{\:\rightharpoonup\:}}
\usepackage{geometry}
\numberwithin{equation}{section}
\DeclareFontEncoding{FMS}{}{}
\DeclareFontSubstitution{FMS}{futm}{m}{n}
\DeclareFontEncoding{FMX}{}{}
\DeclareFontSubstitution{FMX}{futm}{m}{n}
\DeclareSymbolFont{fouriersymbols}{FMS}{futm}{m}{n}
\DeclareSymbolFont{fourierlargesymbols}{FMX}{futm}{m}{n}
\DeclareMathDelimiter{\nr}{\mathord}{fouriersymbols}{152}{fourierlargesymbols}{147}

\DeclareMathDelimiter{\nr}{\mathord}{fouriersymbols}{152}{fourierlargesymbols}{147}
\DeclareMathAlphabet{\mathpzc}{OT1}{pzc}{m}{it}

\title[Relaxed and Inertial Nonlinear forward backward with momentum]{Relaxed and Inertial Nonlinear forward backward with momentum}

\author{Fernando Rold\'an$^{\dagger}$}
\author{Cristian Vega$^\ddag$}
\address{$^{\dagger}$Departamento de Ingeniería Matemática, Universidad de Concepción, Concepción, Chile. {\it 
	E-mail address:} 
	{\sf{fernandoroldan@udec.cl}}. }
\address{$^\ddag$ Instituto de Alta investigaci\'on (IAI), Universidad de Tarapac\'a,
Arica, Chile. {\it 
	E-mail address: } 
	{\sf{cristianvegacereno6@gmail.com}}. }

\begin{document}
\title{Relaxed and Inertial Nonlinear Forward-Backward with Momentum}
\maketitle
\begin{abstract}
In this article, we study inertial algorithms for numerically solving monotone inclusions involving the sum of a maximally monotone and a cocoercive operator. In particular, we analyze the convergence of inertial and relaxed versions of the nonlinear forward-backward with momentum (NFBM). We propose an inertial version including a relaxation step, and a second version considering a double-inertial step with additional momentum. By applying NFBM to specific monotone inclusions, we derive inertial and relaxed versions of algorithms such as forward-backward, forward-half-reflect-backward (FHRB), Chambolle--Pock, Condat--V\~u, among others, thereby recovering and extending previous results from the literature for solving monotone inclusions involving maximally monotone, cocoercive, monotone and Lipschitz, and linear bounded operators. We also present numerical experiments on image restoration, comparing the proposed inertial and relaxation algorithms. In particular, we compare the inertial FHRB with its non-inertial and momentum versions. Additionally, we compare the numerical convergence for larger step-sizes versus relaxation parameters and introduce a  {\it restart} strategy that incorporates larger step-sizes and inertial steps to further enhance numerical convergence.
\end{abstract}

\section{Introduction}
Several applications are modeled through monotone inclusions, such as mechanical problems 
\cite{Gabay1983,Glowinsky1975,Goldstein1964}, differential inclusions \cite{AubinHelene2009,Showalter1997}, convex 
	programming \cite{Combettes2018MP}, 
	game theory \cite{Nash13,BricenoLopez2019}, data science \cite{CombettesPesquet2021strategies}, image processing 
	\cite{BotHendrich2014TV,Briceno2011ImRe}, traffic theory 
\cite{Nets1,Fuku96,GafniBert84}, among other disciplines. For this reason, numerical algorithms for solving monotone inclusions have been extensively studied over the years \cite{BricenoDeride2023,BricenoDerideVega2022,BricenoDavis2018,Briceno2011Skew,ChambollePock2011,Combettes2004Optimization,Condat13,DavisYin2017,Eckstein1992,Lions1979SIAM,Malitsky2020SIAMJO,passty1979JMAA,Tseng2000SIAM,Vu13}. In particular, the nonlinear forward-backward with momentum (NFBM) was proposed in \cite{MorinBanertGiselsson2022} for solving monotone inclusions involving maximally monotone and cocoercive operators.  At each iteration, NFBM involves one evaluation of the cocoercive operator and one evaluation of the {\it warped resolvent}, which is induced by a Lipschitz operator (see also \cite{BuiCombettesWarped2020,Giselsson2021NFBS}), generalizing the classic forward-backward (FB) algorithm \cite{Lions1979SIAM,passty1979JMAA}. Additionally, NFBM allows for variable metrics induced by self-adjoint strongly monotone linear operators.  Consequently, by appropriately selecting the Lipschitz operator and the metric, several methods from the literature, such as forward-half-reflected-backward (FHRB) \cite{Malitsky2020SIAMJO}, Douglas--Rachford \cite{DR1956,Eckstein1992,Lions1979SIAM}, Chambolle--Pock (CP) \cite{ChambollePock2011}, Condat--V\~u (CV)\cite{Condat13,Vu13}, among others, are recovered from NFBM. However, the inertial/relaxed versions of these algorithms are not deduced from NFBM. Inertial and relaxation steps are incorporated into numerical algorithms for solving monotone inclusions in order to accelerate convergence.  The inertial step involves calculating the next iterate using the previous two iterates, while the relaxation step consists of a convex combination of the current and next iterate. For instance, it has been shown in \cite{Alves2020,AlvesMarcavilla2020,Iutzeler2019} that incorporating inertial and/or relaxation steps improves algorithm performance. In this article, we study the convergence of an inertial and relaxed version of NFBM. 

 Many articles studied inertial/relaxed versions of algorithms for solving monotone inclusions. For instance, inertial and relaxation extensions of the gradient descent algorithm have been proposed in \cite{nesterov1983,polyak1964}. Inertial proximal algorithms were studied in \cite{Alvarez2001,AttouchCabot19,AttouchPeypouquet2019,Moudafi2003}. Similarly, the Krasnosel'ski\u{\i}--Mann has been extended with inertial and relaxation steps in \cite{CortildPeypouquet2024,dong2021inertial,dong2022new,MaulenFierroPeypouquet2023}. Inertial versions of FB have been studied in \cite{Alvarez2004,AttouchCabot19,BeckTeboulle2009,LorenzInFB2015,Bot2016NA}. In addition, inertial and relaxed versions of FBF, DR, and CP have been studied in \cite{Bot2016NA,Bot2016Tsengine,BotrelaxFBF2023}, \cite{Alves2020,Bot2015AMC}, and \cite{he2022inertial,valkonen2020inertial}, respectively. 

The main contribution of this work is to incorporate inertial steps in NFBM. In particular, we propose an inertial version of NFBM including a relaxation step, and a second version considering a double-inertial step with additional momentum. For both methods, we prove the weak convergence of its iterates to a solution point of the monotone inclusion. We extend relaxed and inertial versions of algorithms such as FB, FHRB, CP, and CV. Moreover, we recover the existing conditions on the step-sizes that guarantee the weak convergence of the inertial and non-inertial versions of these methods that are available in the literature.  Finally, we present numerical experiments in image restoration for testing the proposed inertial and relaxation algorithms; in particular, we compare the inertial FHRB with its non-inertial and momentum versions. We show that for a fix step-size, the proposed algorithms speed up the convergence in terms of iterations. In the case where the step-size approach to its admissible limit, we introduce a {\it restart} strategy that incorporates both larger step-sizes and inertial steps, obtaining better results than the non-inertial version in iterations and CPU time.

This article is organized as follows: In Section \ref{sec:Pre}, we introduce the notation and mathematical background. Section \ref{sec:main} provides the convergence analysis of both inertial algorithms. In Section~\ref{sec:PC} we present inertial/relaxed algorithms deduced from our main result.  Section \ref{se:NE} is dedicated to
numerical experiments in image restoration.
Finally, the conclusions are presented in Section~\ref{se:conc}.
\section{Problem Statement and Proposed Algorithm}\label{sec:Pro}
In this article we aim to solve numerically the following problem.
\begin{pro}\label{pro:main}
	Let $\H$ be a real Hilbert space, let $A: \H \to 2^\H$ be a maximally monotone operator and let	$C:\H \to \H$ be a 
	$\mu$-cocoercive operator for $\mu \in \RPP$. The problem is to
	\begin{equation*} 
		\text{find }  x \in \H  \text{ such that } 
		0 \in (A+C)x 
	\end{equation*}
	under the hypothesis that its solution set, denoted by 
	$\bm{Z}$, is not empty.
\end{pro}
This problem can be solved by the NFBM, which for starting points $(x_0,u_{0}) \in \H^2$, iterates as follows: 
\begin{equation}
	\label{eq:algoNFB}
	(\forall n \in \N)\quad \left\lfloor
	\begin{aligned}
		&x_{n+1} = (M_{n}+A)^{-1} \left(M_{n}x_{n}-Cx_{n}+u_{n}/\gamma_{n}\right),\\
		&u_{n+1} = (\gamma_{n} M_{n} - S) x_{n+1}- (\gamma_{n} M_{n} - S) x_{n}.
	\end{aligned}
	\right.
\end{equation} 
where, for every $n \in \N$, $\gamma_{n} \in \RPP$, $M_{n} \colon\H 
\to \H$ is such that $\gamma_{n}M_{n}-S$ is $\zeta_{n}$-Lipschitz and $S\colon \H \to \H$ is a self-adjoint strongly monotone linear operator. In the case when, for every $n \in \N$, $\gamma_{n}=\gamma$, $\gamma_{n}M_{n}=\id$ and $S=\id$, the recurrence in \eqref{eq:algoNFB} corresponds to FB. Moreover, if we set $M_{n}= \id/\gamma - B$, $S = \id$, and $A=\tilde{A}+B$, where $\gamma \in \RPP$, $\tilde{A} \colon \H \to 2^\H$ is a set valued operator, $B \colon \H \to \H$ is a $\zeta$-Lipschitzian operator for $\zeta \in \RPP$ and $A$ is maximally monotone, \eqref{eq:algoNFB} reduces to
\begin{equation}
	\label{eq:algoMT}
	(\forall n \in \N)\quad \left\lfloor
	\begin{aligned}
		&x_{n+1} = J_{\gamma A} \left(x_{n}-\gamma(2Bx_{n}-Bx_{n-1}+Cx_{n}\right)).
	\end{aligned}
	\right.
\end{equation} 
Note that the recurrence in \eqref{eq:algoMT} corresponds to FHRB and $(x_{n})_{n \in \N}$ converges weakly to a zero of $(\tilde{A}+B+C)$ if $\gamma \in ]0,2\mu/(4\zeta\mu+1)[$. On the other hand, primal-dual methods generalizing CP and CV are deduced from the recurrence in \eqref{eq:algoNFB} by an adequate choice of $M_{n}$ and $S$ (see \cite[Section~6]{MorinBanertGiselsson2022}).  Therefore, despite our main problem consisting in an inclusion that involves only two operators, the recurrence in \eqref{eq:algoMT} can be adapted for finding zeros of monotone inclusions involving the sum of several operators, as in \cite[Section~6]{MorinBanertGiselsson2022}.

In this article, we study the convergence of the following inertial and relaxed version of \eqref{eq:algoNFB}.
\begin{algo}\label{algo:NFBInertial}
	In the context of Problem~\ref{pro:main}, let $(\gamma_{n})_{n \in \N}$ be a sequence in $\RPP$, let $(\alpha_{n})_{n \in \N}$ be a sequence in $[0,1]$, let $\lambda \in ]0,2[$, let $(x_0,x_{-1},u_{0}) \in \H^{3}$, and consider the sequence defined recursively by
	\begin{equation}
		\label{eq:algoNFBInertial}
		(\forall n \in \N)\quad \left\lfloor
		\begin{aligned}
			&y_{n} = x_{n}+\alpha_{n} (x_{n}-x_{n-1}),\\
			&p_{n+1} = (M_{n}+A)^{-1} \left(M_{n}y_{n}-Cy_{n}+u_{n}/\gamma_{n}\right),\\
			&u_{n+1} = (\gamma_{n} M_{n} - S) p_{n+1}- (\gamma_{n} M_{n} - S) y_{n},\\ &x_{n+1}=(1-\lambda)y_{n}+\lambda p_{n+1}.
		\end{aligned}
		\right.
	\end{equation} 
\end{algo}
Note that, in the case where $\alpha_{n}\equiv 0$ and $\lambda = 1$, \eqref{eq:algoNFBInertial} reduces to \eqref{eq:algoNFB}. Furthermore, in the same setting mentioned above, from \eqref{eq:algoNFBInertial}, we deduce an inertial and relaxed version of FHRB (see \eqref{eq:algoIMT}).
 Versions of NFBM and FHRB with additional momentum have been proposed in \cite{MorinBanertGiselsson2022} and \cite{Malitsky2020SIAMJO,Tang2022InFRB}, respectively. However, these extensions are not completely inertial and do not cover the recurrences in \eqref{eq:algoNFBInertial} and \eqref{eq:algoIMT}.  In order to recover these extensions, we provide the following algorithm that additionally incorporates momentum and a second relaxation step. The second inertial step allows for more flexibility in the choice of inertial parameters. For simplicity, we do not consider the relaxation step in this recurrence. 
\begin{algo}\label{algo:NFBDoubleInertial}
	In the context of Problem~\ref{pro:main}, let $(\alpha_{n})_{n \in \N}$,  $(\beta_{n})_{n \in \N}$, and $(\theta_{n})_{n \in \N}$ sequences in $[0,1]$, let $(x_0,x_{-1}, u_{0}) \in \H^{3}$, and consider the sequence defined recursively by
\begin{equation}
		\label{eq:algoNFBDoubleInertial}
		(\forall n \in \N) \left\lfloor
\begin{aligned}
			&y_{n} = x_{n}+\alpha_{n} (x_{n}-x_{n-1}),\\
            &z_{n} = x_{n}+\beta_{n} (x_{n}-x_{n-1}),\\
			&x_{n+1} = (M_{n}+A)^{-1} \left(M_{n}y_{n}-Cz_{n}+u_{n}/\gamma_{n}+\theta_{n} S(x_{n}-x_{n-1})/\gamma_{n}\right),\\
   &u_{n+1} = (\gamma_{n} M_{n} - S) x_{n+1}- (\gamma_{n} M_{n} - S) y_{n}.
		\end{aligned}
  \right.
	\end{equation} 
\end{algo}

\section{Notation and Preliminaries} \label{sec:Pre}
In this paper, $\H$ and $\G$ are real Hilbert spaces with scalar product $\scal{\cdot}{\cdot}$ and norm $\|\cdot \|$. We denoted by $\to$ the strong convergence and $\weakly$ the weak convergence. The identity operator is denoted by $\id$. Given a linear operator $L:\H\to\G$, we denote its adjoint by $L^* \colon \G\to\H$ and its norm by $\|L\|$. $L$ is a self-adjoint operator if $L^*=L$ and is strongly monotone if there exists $\alpha \in \RPP$ such that, for every $x\in \H$, $\scal{Lx}{x}\geq \alpha \|x\|^2$. Henceforth, $S\colon \H \to \H$ is a self-adjoint strongly monotone linear operator.  The inner product and norm induced by $S$ are denoted by $\scal{\cdot}{\cdot}_S:=\scal{S\cdot}{\cdot}$ and $\|\cdot \|_S$, respectively. Hence, for every $(x,y,z) \in \H^3$ and $\alpha \in \R$
\begin{align}
\label{eq:prodinternonorma} 2\scal{x-y}{y-z}_S &= \|x-z\|_S^2-\|x-y\|^2_S-\|y-z\|^2_S,\\
  \label{eq:normaalpha}  \|\alpha x + (1-\alpha)y\|^2_S &= \alpha\|x\|^2_S+(1-\alpha)\|y\|^2_S-\alpha(1-\alpha)\|x-y\|^2_S.
 \end{align}
Note that, there exists a self-adjoint strongly monotone linear operator $S^{1/2} \colon \H \to \H$ such that $S=S^{1/2}\circ S^{1/2}$. Therefore, the Cauchy--Schwarz inequality can be extended to the norm induced by $S$ and $S^{-1}$ in the following sense, for every $(x,u) \in 
		\H^2$
 \begin{align*}
		|\scal{x}{u}| 
  = \left|\scal{S^{-1/2}x}{S^{1/2}u}\right|
  \leq \|S^{-1/2}x\|\|S^{1/2}u\|
  = \|x\|_{S^{-1}}\|u\|_S.
 \end{align*}
	Let $T\colon \H \rightarrow 
	\H$ and $\beta \in \left]0,+\infty\right[$. The operator $T$ is 
	$\beta$-cocoercive with respect to $S$ if for every $(x,y) \in \H^2$, $\scal{x-y}{Tx-Ty}
		\geq \beta \|Tx - Ty 
		\|^2_{S^{-1}}$. The operator $T$ is $\beta$-Lipschitzian with respect to $S$ if $|Tx-Ty\|_{S^{-1}} \leq  
		\beta\|x - y \|_{S}$.
	Let $A\colon\H \rightarrow 2^{\H}$ be a set-valued operator.
	The graph of $A$ 
	is defined by 
     $\gra A = \menge{(x,u) \in \H \times \H}{u \in Ax}$
 and the set of of zeros of $A$ is given by 
 $\zer A = 	\menge{x \in \H}{0 \in Ax}$.
  The inverse of the operator $A$ is defined by
	$A^{-1} \colon u \mapsto \menge{x \in \H}{u \in Ax}$. The operator $A$ is called monotone if for all $\big((x,u),(y,v)\big) \in (\gra A)^2$, $\scal{x-y}{u-v} \geq 0$.
	Moreover, $A$ is maximally monotone if it is 
	monotone and its graph is 
	maximal in the sense of 
	inclusions among the graphs of monotone operators. 	The resolvent of a maximally monotone operator $A$ is 
	defined by $J_A:=(\id+A)^{-1}$.
    Note that $J_A$ is single valued. If $A$ is maximally monotone, then its inverse $A^{-1}$ is also a maximally monotone operator. We denote by $\Gamma_0(\H)$ the class of proper lower 
semicontinuous convex functions $f\colon\H\to\RX$. Let 
$f\in\Gamma_0(\H)$.
The Fenchel conjugate of $f$ is 
defined by $f^*\colon u\mapsto \sup_{x\in\H}(\scal{x}{u}-f(x))$ 
and we have
$f^*\in \Gamma_0(\H)$. The subdifferential of $f$ is the 
maximally monotone operator
$\partial f\colon x\mapsto \menge{u\in\H}{(\forall y\in\H)\:\: 
	f(x)+\scal{y-x}{u}\le f(y)}$, we have that
$(\partial f)^{-1}=\partial f^*$ 
and that $\zer\,\partial f$ is the set of 
minimizers of $f$, which is denoted by $\arg\min_{x\in \H}f$. 	
We denote the proximity operator of $f$ by $\prox_f=J_{\partial f}$.
For further background on monotone operators and convex analysis, 
the reader is referred to \cite{bauschkebook2017}.
\section{Inertial Nonlinear Forward-Backward with Momentum Correction}\label{sec:main}
This section is divided in two parts. First, we present the convergence analysis of Algorithm~\ref{algo:NFBInertial}. Next, we derive the convergence of  Algorithm~\ref{algo:NFBDoubleInertial}.
\subsection{Convergence of Algorithm~\ref{algo:NFBInertial}}
The following assumption allows us to guarantee the convergence of the proposed algorithm, it was introduced \cite[Assumption~2.2 \& Proposition~2.1]{MorinBanertGiselsson2022}.
\begin{asume}\label{asume:1}
    In the context of Problem~\ref{pro:main}, let $(\underline{\gamma},\overline{\gamma}) \in \RPP^2$, let $(\gamma_{n})_{n \in \N}$ be a sequence in $[\underline{\gamma},\overline{\gamma}]$, let $S\colon \H \to \H$ be a strongly monotone self-adjoint linear bounded operator, and, for every $n \in \N$,  let $M_{n}:\H 
	\to \H$ be such that $\gamma_{n}M_{n}-S$ is $\zeta_{n}$-Lipschitz with respect to $S$ for $\zeta_{n} \in [0,1-\varepsilon]$ and $\varepsilon \in ]0,1[.$
\end{asume}
The following proposition is a previous result that will be used to prove the convergence of the proposed method.
\begin{prop}
 \label{prop:descn}
  In the context of Problem~\ref{pro:main} and Assumption~\ref{asume:1}, consider the sequence $(x_{n})_{n \in \N}$ defined recursively by Algorithm~\ref{algo:NFBInertial} with initialization points $(x_0,x_{-1},u_0)\in \H^3$. Let $x \in \bm{Z}$ and, for every $n \in \N$, define
  \begin{align}
        \label{eq:defTn}
        T_{n}=&\gamma_{n}M_{n}-S,\\ \label{def:nun}
        \nu_n =& \begin{cases}
        \zeta_{n-1}, \text{ if } -T_n \text{ is monotone and } \lambda \geq 1,\\
        2\zeta_n+\zeta_{n-1}, \text{ otherwise.}
        \end{cases}\\
        \label{eq:defrho} \rho_{n}=&\left(2-\lambda-|1-\lambda|\nu_n -\frac{\gamma_{n}}{2\mu}-(1+|1-\lambda|)\zeta_{n}\right),\\ 
      \label{eq:defetan}\eta_{n} =& (1-\alpha_{n})\frac{\rho_{n}}{\lambda}-\lambda\zeta_{n-1},\\
      \label{eq:defxin}\xi_{n} =& \alpha_{n}\left(1+\alpha_{n}\right)+\alpha_{n}(1-\alpha_{n})\frac{\rho_{n}}{\lambda},\\
    \label{eq:defCn} C_{n+1}(x)=& \|x_{n+1}-x\|^{2}_{S} -\alpha_{n}\|x_{n}-x\|^2_S+2\lambda \scal{u_{n+1}}{x_{n+1}-x}\nonumber\\&+\lambda(1+|1-\lambda|)\zeta_{n}\|p_{n+1}-y_{n}\|^2_S+\xi_{n+1}\|x_{n+1}-x_{n}\|^{2}.
\end{align}
Suppose that there exists $N_0 \in \N$ such that $(\alpha_{n})_{n\geq N_0}$ is non-decreasing, and that
\begin{equation}\label{eq:desrho}
    (\forall n \geq N_0) \quad \rho_{n} \geq 0.
\end{equation}
Then, the following hold:
\begin{enumerate}
\item\label{prop:descn0} For every $n\geq N_0$, $\xi_{n}$ is a non-negative.
    \item\label{prop:descn1} For every $n \in \N$,
    \begin{equation}\label{eq:descn1} 
    C_{n+1}(x)\leq  C_{n}(x)
	-(\eta_{n}-\xi_{n+1})\|x_{n+1}-x_{n}\|^{2}.
    \end{equation}
\end{enumerate}
Moreover, suppose that there exists  $\epsilon \in \RPP$ such that 
\begin{equation}\label{eq:desetaxi}
	 (\forall n \geq N_0) \quad \eta_{n}-\xi_{n+1}\geq \epsilon
	\end{equation}
	and either $\lambda \in [1,2[$ or $\lambda \in ]0,1[$ and $(\xi_{n})_{n \geq N_0}$ is non-decreasing. Then,
\begin{enumerate}[resume]
\item\label{prop:descn2} For every $n_0 \geq  N_0$, $-\alpha_{n-1} \geq - \left( 1- \frac{\lambda\zeta_{n-1}}{1+|1-\lambda|} \right)$.    
\item\label{prop:descn3} $(C_{n}(x))_{n \geq N_0}$  is a non-negative convergent sequence.
\item\label{prop:descn4} $\sum_{n \in \N} \|x_{n+1}-x_{n}\|^2 < +\infty$.
\end{enumerate}
\end{prop}
 \begin{proof}
Fix $n \geq  N_0$.
\begin{enumerate} 
     \item Since $\alpha_{n} \in [0,1]$ and $\rho_{n} \geq 0$, the result follows from \eqref{eq:defxin}.
    \item   It follows from \eqref{eq:algoNFBInertial} that
    	\begin{align*}
		\gamma_{n} M_{n}y_{n}&-\gamma_{n} Cy_{n}+u_{n}-\gamma_{n} M_{n} p_{n+1}\in \gamma_{n} A p_{n+1}\nonumber\\
  &\Leftrightarrow 
  Sy_{n} + u_{n}  - (Sp_{n+1}+u_{n+1})-\gamma_{n} Cy_{n} \in \gamma_{n} A p_{n+1}.
	\end{align*}
  Let $x \in \bm{Z}$, thus, $-\gamma_{n} Cx \in \gamma_{n} A x$. Then, by monotonicity of $A$, we have 
  \begin{equation}\label{eq:monA}
         2\scal{Sy_{n} + u_{n}  - (Sp_{n+1}+u_{n+1})-\gamma_{n} Cy_{n}+\gamma_{n} Cx}{p_{n+1}-x}\geq 0.
\end{equation}
Let us bound the terms in \eqref{eq:monA}. First, \eqref{eq:prodinternonorma} and  \eqref{eq:normaalpha} yield
\begin{align}\label{eq:proofp1}
2\scal{y_{n} - p_{n+1}}{p_{n+1}-x}_S
    =& \|y_{n}-x\|_S^2-\|y_{n}-p_{n+1}\|_S^2-\|p_{n+1}-x\|_S^2\nonumber\\
    =& (1+\alpha_{n})\| x_{n}-x\|_S^{2}+\alpha_{n}(1+\alpha_{n}) \|x_{n}-x_{n-1}\|_S^{2}\nonumber\\
    &-\alpha_{n}\|x_{n-1}-x\|_S^{2}-\|y_{n}-p_{n+1}\|_S^2\nonumber\\
    &\quad -\|p_{n+1}-x\|_S^2.
\end{align}
Now, it follows from \eqref{algo:NFBInertial} that
\begin{align}\label{eq:proofp2}
       &2\scal{u_{n} -u_{n+1}}{p_{n+1}-x}\nonumber\\
       &\hspace{0.5cm}=2\scal{u_{n} -u_{n+1}}{x_{n+1}-x}+  2\scal{u_{n} -u_{n+1}}{p_{n+1}-x_{n+1}}\nonumber\\
               &\hspace{0.5cm}=2\scal{u_{n} -u_{n+1}}{x_{n+1}-x}+2(1-\lambda)\scal{u_{n} -u_{n+1}}{p_{n+1}-y_{n}}
\end{align}
\noindent Since $T_{n}$ is $\zeta_{n}$-Lipschitz, decomposing the first term, we have that
\begin{align}\label{eq:proofp22}
2&\scal{u_{n}}{x_{n+1}-x} \nonumber\\
&= 2\scal{u_n}{x_{n}-x}+2\scal{u_n}{x_{n+1}-x_{n}}\nonumber\\
&\leq  2\scal{u_n}{x_{n}-x}+2\left(\|u_{n}\|_{S^{-1}}\|{x_{n+1}-x_{n}}\|_S\right)\nonumber\\
&\leq  2\scal{u_n}{x_{n}-x}+\zeta_{n-1}\left(\|p_{n}-y_{n-1}\|^2_S+\|{x_{n+1}-x_{n}}\|^2_S\right).
\end{align}
Developing the second term in \eqref{eq:proofp2}, by \eqref{def:nun}, we obtain
\begin{align}
    2(1-\lambda)&\scal{u_{n} -u_{n+1}}{p_{n+1}-y_{n}}\nonumber
    \\
    =&2(1-\lambda)(\scal{u_n}{p_{n+1}-y_{n}}-\scal{u_{n+1}}{p_{n+1}-y_{n}})\nonumber\\ \leq& |1-\lambda|(\nu_n\|p_{n+1}-y_{n}\|_{S}^{2}+\zeta_{n-1}\|p_{n}-y_{n-1}\|_{S}^{2}).
\end{align}
Moreover, by the $\mu$-cocoercivity of $C$  we have
	\begin{align}\label{eq:proofp3}
	2&\gamma_{n}\scal{Cy_{n}-Cx}{x-p_{n+1}}\nonumber\\
    &= 
	2\gamma_{n}\scal{Cy_{n}-Cx}{x-y_{n}}+2\gamma_{n}\scal{Cy_{n}-Cx}{y_{n}-p_{n+1}}\nonumber\\
	&\leq
	-2\gamma_{n}\mu\|Cy_{n}-Cx\|^{2}_{S^{-1}}+2\gamma_{n}\mu\|Cy_{n}-Cx\|^{2}_{S^{-1}}
	+\frac{\gamma_{n}}{2\mu}\|p_{n+1}-y_{n}\|^{2}_{S}\nonumber\\
	&=\frac{\gamma_{n}}{2\mu}\|p_{n+1}-y_{n}\|^{2}_S.
\end{align}
Hence, it follows from \eqref{eq:monA}-\eqref{eq:proofp3} that
\begin{align}\label{eq:proofp4}
    &\|p_{n+1}-x\|^2_S+2\scal{u_{n+1}}{x_{n+1}-x} \nonumber\\
    &\leq(1+\alpha_{n})\|x_{n}-x\|^2_S-\alpha_{n}\|x_{n-1}-x\|^2_S+2\scal{u_n}{x_{n}-x}+\zeta_{n-1}\|p_{n}-y_{n-1}\|^2_S\nonumber\\
    &  +\alpha_{n}(1+\alpha_{n}) \|x_{n}-x_{n-1}\|^2_S - \|p_{n+1}-y_{n}\|^2_S+\frac{\gamma_{n}}{2\mu}\|p_{n+1}-y_{n}\|^2_S  \nonumber \\ 
    & + \zeta_{n-1}\|x_{n+1}-x_{n}\|^{2}_{S}+\nu_n|1-\lambda|\|p_{n+1}-y_{n}\|_{S}^{2}+|1-\lambda|\zeta_{n-1}\|p_{n}-y_{n-1}\|_{S}^{2}
\end{align}
Moreover, it follows from \eqref{eq:normaalpha} that
\begin{align}\label{eq:proofp5}
    \|x_{n+1}-x\|^{2}_{S}=&(1-\lambda)\|y_{n}-x\|_{S}^{2}+\lambda\|p_{n+1}-x\|_{S}^{2}-\lambda(1-\lambda)\|p_{n+1}-y_{n}\|_{S}^{2}\nonumber\\ =&(1-\lambda)(1+\alpha_{n})\|x_{n}-x\|_{S}^{2}-(1-\lambda)\alpha_{n}\|x_{n-1}-x\|_{S}^{2}\nonumber\\&+(1-\lambda)\alpha_{n}(1+\alpha_{n})\|x_{n}-x_{n-1}\|_{S}^{2}+\lambda\|p_{n+1}-x\|_{S}^{2}\nonumber\\&-\lambda(1-\lambda)\|p_{n+1}-y_{n}\|_{S}^{2}. 
\end{align}
Then, combining \eqref{eq:proofp4} and  \eqref{eq:proofp5}, we deduce
\begin{align}\label{eq:proofp6}
    &\|x_{n+1}-x\|^{2}_{S} -\alpha_{n}\|x_{n}-x\|^2_S+2\lambda \scal{u_{n+1}}{x_{n+1}-x} \nonumber\\
    &\leq\|x_{n}-x\|^2_S-\alpha_{n}\|x_{n-1}-x\|^2_S+2\lambda\scal{u_n}{x_{n}-x}+ \lambda\zeta_{n-1}\|x_{n+1}-x_{n}\|^{2}_{S}\nonumber\\
    &\quad  +\lambda(1+|1-\lambda|)(\zeta_{n-1}\|p_{n}-y_{n-1}\|^2_S
    -\zeta_{n}\|p_{n+1}-y_{n}\|^2_S)\nonumber\\
    &\quad+\alpha_{n}(1+\alpha_{n})\|x_{n}-x_{n-1}\|^2_S-\lambda\rho_{n}\|p_{n+1}-y_{n}\|^{2}_{S}.
\end{align}
Observe that $p_{n+1}-y_{n}=\lambda^{-1}(x_{n+1}-y_{n})$. Then, in view of \eqref{eq:normaalpha} and \eqref{eq:desrho}, the last term in \eqref{eq:proofp6} can be bound as follows: 
\begin{align}\label{eq:proofp7}
-\lambda\rho_{n}\|p_{n+1}-y_{n}\|^{2}_{S}
&=-\frac{\rho_{n}}{\lambda}\|(1-\alpha_{n})(x_{n+1}-x_{n})+\alpha_{n}(x_{n+1}-2x_{n}+x_{n-1})\|^{2}\nonumber\\
&=-\frac{\rho_{n}}{\lambda}\Big((1-\alpha_{n})\|x_{n+1}-x_{n}\|^{2}-\alpha_{n}(1-\alpha_{n})\|x_{n}-x_{n-1}\|^{2}\nonumber\\
& \hspace{1.2cm}+\alpha_{n}\| x_{n+1}-2x_{n}+x_{n-1}\|^{2}\Big)\nonumber\\
&\leq-\frac{(1-\alpha_{n})\rho_{n}}{\lambda}(\|x_{n+1}-x_{n}\|^{2} - \alpha_{n}\|x_{n}-x_{n-1}\|^{2}).
\end{align}
Combining \eqref{eq:proofp6} and \eqref{eq:proofp7} and rearranging terms, we obtain
\begin{align*}
        &\|x_{n+1}-x\|^{2}_{S} -\alpha_{n}\|x_{n}-x\|^2_S+2\lambda \scal{u_{n+1}}{x_{n+1}-x} \\
    &\leq\|x_{n}-x\|^2_S-\alpha_{n-1}\|x_{n-1}-x\|^2_S+2\lambda\scal{u_n}{x_{n}-x}+\xi_n\|x_{n}-x_{n-1}\|^2_S\\
    &+\lambda{ (1+|1-\lambda|)}(\zeta_{n-1}\|p_{n}-y_{n-1}\|^2_S-\zeta_{n}\|p_{n+1}-y_{n}\|^2_S)
    - \eta_n\|x_{n+1}-x_{n}\|^{2}_{S},
\end{align*}  
which is equivalent to \eqref{eq:descn1}.
\item First, suppose that $\lambda \in [1,2[$, we have $ \lambda^2+\lambda-2 \geq 0$, thus $\lambda^2 \geq 2- \lambda \geq \rho_{n}$, in view of \eqref{eq:defrho}. It follows from \eqref{eq:desetaxi} that
\begin{align*}
	0\leq \eta_{n} \leq \xi_{n+1}=(1-\alpha_{n})\lambda -\lambda \zeta_{n-1} \Rightarrow \alpha_{n}\leq 1-\zeta_{n-1}.
\end{align*}
Then, the non-decreasing property of $(\alpha_{n})_{n \geq N_0}$  yields
\begin{align*}
	-\alpha_{n-1}\geq-\alpha_{n}\geq -(1-\zeta_{n-1}) = -\left(1-\frac{\lambda\zeta_{n-1}}{1+|1-\lambda|}\right).
\end{align*}
Now, suppose that $\lambda \in ]0,1[$ and  that $(\xi_{n})_{n \geq N_0}$ is non-decreasing. Thus, $\lambda\eta_{n}\geq \lambda\xi_{n+1}\geq \lambda\xi_{n}\geq \alpha_{n}(1-\alpha_{n})\rho_{n}$. Hence, in view of \eqref{eq:defetan} and \eqref{eq:defrho}
\begin{align}\label{eq:desalphazeta1}
\lambda^2 \zeta_{n-1}&\leq  (1-\alpha_{n})^2\rho_{n}\leq  (1-\alpha_{n})^2(2-\lambda)\leq  (1-\alpha_{n})^2(2-\lambda)^2.
\end{align}	
Moreover, since $\zeta_{n-1}<1$, we have that $\zeta_{n-1}<\sqrt{\zeta_{n-1}}$. Then, by taking square root in \eqref{eq:desalphazeta1} and by the non-decreasing property of $(\alpha_{n})_{n \geq N_0}$ we deduce
\begin{align*}
	-\alpha_{n-1}\geq -\alpha_{n} \geq -\left(1-\frac{\lambda\zeta_{n-1}}{2-\lambda}\right)=-\left(1-\frac{\lambda\zeta_{n-1}}{1+|1-\lambda|}\right).
\end{align*}

\item Since, for every $n \geq N_0 $, $\eta_{n}-\xi_{n+1}\geq\epsilon$, in view of \eqref{eq:descn1} we conclude that $(C_{n}(x))_{n\geq N_0}$ is non-increasing. To show that it is non-negative, suppose that $C_{n_{1}}(x) < 0$ for some $n_{1} \geq N_0$. Since $(C_{n}(x))_{n\geq N_0}$ is non-increasing, for every $n\geq n_1$, we have 
$0>C_{n_{1}}(x)\geq C_{n}(x)$. Moreover, for every $n\geq n_1$,
\begin{align}
   C_{n}(x)   \geq&  \|x_{n}-x\|^{2}-\alpha_{n-1}\|x_{n-1}-x\|^2 +  \lambda{ (1+|1-\lambda|)}\zeta_{n-1}\|x_{n}-y_{n-1}\|^2\nonumber\\
    &\quad \quad  -  \lambda{ (1+|1-\lambda|)}\zeta_{n-1}\|x_{n}-y_{n-1}\|^2 -\frac{\lambda\zeta_{n-1}}{1+|1-\lambda|}\|x_{n}-x\|^2\nonumber\\\label{eq:proofp95}
    = & \left(1-\frac{\lambda\zeta_{n-1}}{1+|1-\lambda|}\right)\|x_{n}-x\|^{2}-\alpha_{n-1}\|x_{n-1}-x\|^2\\ 
    \geq & \left(1-\frac{\lambda\zeta_{n-1}}{1+|1-\lambda|}\right)(\|x_{n}-x\|^{2}-\|x_{n-1}-x\|^2),\nonumber
\end{align}
where the last inequality follows from \ref{prop:descn2}. Hence, for every $n\geq n_1$,
\begin{equation*}
\|x_{n} - x\|^2 \leq \|x_{n-1} - x\|^2 + \frac{C_{n_1}(x)}{1-\frac{\lambda\zeta_{n-1}}{1+|1-\lambda|}}\leq \|x_{n-1} - x\|^2 + C_{n_1}(x) .
\end{equation*}
Therefore, for every $n\geq n_1$,
\begin{equation*}0 \leq \|x_{n} - x\|^2 \leq \|x_{n-1} - x\|^2 + C_{n_1}(x)\leq \cdots \leq \|x_{n_1} - x\|^2 + (n - n_1)C_{n_1}(x),
\end{equation*}
which leads to a contradiction. Consequently, $(C_{n}(x))_{n \geq N_0}$ is non-negative and convergent. 
\item Since, for every $n \in \N$, $\eta_{n}-\xi_{n+1}\geq\epsilon$, the result follows from \eqref{eq:descn1} and \cite[Lemma 5.31]{bauschkebook2017}.
\end{enumerate}
\end{proof}
\begin{rem}
  \begin{enumerate}
      \item The non-decreasing assumption on $(\xi_{n})_{n \in \N}$ is  satisfied when $\alpha_n\equiv \alpha $ and $\zeta_n \equiv \zeta$. Furthermore, this assumption is required only when $\lambda \in ]0,1[$. In general, the best convergence results for relaxed algorithms are achieved when $\lambda >1$ (see, for instance, \cite[Section~5]{MaulenFierroPeypouquet2023}). 
      \item In the case where $\lambda \geq 1$ and $-T_n$ is monotone for every $n \in \N$, we have $\nu_n = \zeta_{n-1}$, which provides more flexibility in the choice of $\alpha_n$ and $\lambda$ to satisfy \eqref{eq:desrho} and \eqref{eq:desetaxi}. It is worth noting that in all the examples presented in Section~\ref{sec:PC},  $-T_n$ is monotone.
  \end{enumerate}
\end{rem}
The following theorem establishes the weak convergence of Algorithm~\ref{algo:NFBInertial} to a solution to Problem~\ref{pro:main}.
\begin{teo}\label{teo:main}
  In the context of Problem~\ref{pro:main} and Assumption~\ref{asume:1}, consider the sequence $(x_{n})_{n \in \N}$ defined recursively by Algorithm~\ref{algo:NFBInertial} with initialization points $(x_0,x_{-1}, u_{0})\in \H^{3}$. Let $(\rho_{n})_{n \in \N}$, $(\eta_{n})_{n \in \N}$, and $(\xi_{n})_{n \in \N}$, be the sequences defined in \eqref{eq:defrho}, \eqref{eq:defetan}, and \eqref{eq:defxin}, respectively.
Suppose that there exist $N_0 \in \N$ such that $(\alpha_{n})_{n \geq N_0}$ is non-decreasing and  $\epsilon \in \RPP $ such that
\begin{equation}\label{eq:desalphabeta2}
   (\forall n \geq N_0) \quad \quad \rho_{n}\geq 0 \text{ and } \eta_{n}-\xi_{n+1}\geq \epsilon.
\end{equation}
Moreover, suppose that either $\lambda \in [1,2[$ or $\lambda \in ]0,1[$ and $(\xi_{n})_{n \geq N_0}$ is non-decreasing. Then, $(x_{n})_{n \in \N}$ converges weakly to a point in $\bm{Z}$.
\end{teo}
 \begin{proof} Let $x\in \bm{Z}$. By Proposition~\ref{prop:descn} we conclude that $(C_{n}(x))_{n \in \N}$ is non-negative and convergent, and $\sum_{n \in \N} \|x_{n+1}-x_{n}\|^2<+\infty.$
Since, for every $n \in \N$, $y_{n} = x_{n} + \alpha_{n} (x_{n} - x_{n-1})$, $x_{n+1}-y_{n} = \lambda(p_{n+1}-y_{n})$, $u_{n+1}=T_{n}p_{n+1}-T_{n}y_{n}$, $T_{n}=\gamma_{n} M_{n}-S$ is $(1-\varepsilon)$-Lipschitz, and $C$ is $(1/\mu)$-Lipschitz, we obtain that
\begin{equation}\label{eq:proofc1}
    \|x_{n+1}-y_{n}\|\to 0, \quad  \|p_{n+1}-y_{n}\|\to 0, \quad
\|u_{n+1}\|\to 0, \text{ and } \|Cp_{n+1}-Cy_{n}\|\to 0.    
    \end{equation}
Now, let $\delta\in [\tau +\inf\alpha_{n-1}, \sup (1-\zeta_{n})-\tau]$ for $\tau \in \RPP$ small enough, thus, by applying the Young's Inequality with parameter $(\delta-\alpha_{n-1})/\alpha_{n-1}$
\begin{align}\label{eq:proof6}
    \alpha_{n-1} \|x_{n-1} - x\|^2 & = \alpha_{n-1}(\|x_{n} - x\|^2 -2\scal{x_n-x}{x_n-x_{n-1}}+\|x_n-x_{n-1}\|^2)\nonumber\\
    &\leq \delta\|x_{n} - x\|^2 + \frac{\alpha_{n-1}\delta}{\delta- \alpha_{n-1}} \|x_{n} - x_{n-1}\|^2.
\end{align}
Moreover, by \eqref{eq:proof6} and \eqref{eq:proofp95} we have that
\begin{align*}
    \left(1 - \zeta_{n-1} - \delta\right)\|x_{n} - x\|^2 \leq &\left(1 -\zeta_{n-1}\right)\|x_{n} - x\|^2 - \alpha_{n-1}\|x_{n-1} - x\|^2 \\
    & + \frac{\alpha_{n-1}\delta}{\delta- \alpha_{n-1}}\|x_{n} - x_{n-1}\|^2 \\\leq &C_{n}(x) + \frac{\alpha_{n-1}\delta}{\delta- \alpha_{n-1}}\|x_{n} - x_{n-1}\|^2.
\end{align*}
Since $(C_{n}(x))_{n \in \N}$ and $(\|x_{n} - x_{n-1}\|)_{n \in \N}$ are convergent and $\left(\frac{\alpha_{n-1}\delta}{\delta-\alpha_{n-1}}\right)_{n \in \N}$ is bounded, $(\|x_{n} - x\|)_{n \in \N}$ is also bounded. Set $M=\sup_{n \in \N} \|x_{n}-x\|$. Therefore, in view of \eqref{eq:proofc1} we conclude that
\begin{align}\label{eq:proofc2}
    |\|x_{n} - x\|^2 - \|x_{n-1} - x\|^2|=& |\|x_{n} - x\|- \|x_{n-1} - x\||(\|x_{n} - x\| + \|x_{n-1} - x\|)\nonumber\\ \leq& 2M|\|x_{n} - x\|- \|x_{n-1} - x\||\nonumber\\ \leq& 2M|\|x_{n} - x_{n-1} \||\rightarrow 0.
\end{align}
Since $(\alpha_{n})_{n \geq N_0}$ is non-decreasing and bounded, it converges. Moreover, by \eqref{eq:desetaxi} and \eqref{eq:defetan} we conclude that $(1-\alpha_{n-1}) \geq \lambda \epsilon/2>0$. Then,  for every $x \in \bm{Z}$, the convergence of $(\|x_{n}-x\|)_{n \in \N}$, is deduced by \eqref{eq:proofc1} and \eqref{eq:proofc2} noticing that
\begin{align}
     \lim_{n\to \infty}\left(1-\alpha_{n-1}\right)\|x_{n}-x\|^2=\lim_{n\to \infty}\Big(&C_{n}(x)-2\scal{u_n}{x-x_{n}} -\xi_{n}\|x_{n}-x_{n-1}\|^2\nonumber\\&-\alpha_{n-1}\left(\|x_{n}-x\|^2-\|x_{n-1}-x\|^2\right)\nonumber\\&-\lambda(1+|1-\lambda|)\zeta_{n-1}\|p_{n}-y_{n-1}\|^2_S\Big).\nonumber    
\end{align}
Note that, by \eqref{eq:algoNFBInertial}, we have 
		\begin{align}
			 M_{n}y_{n}+u_{n}/\gamma_{n}-Cy_{n}-M_{n}p_{n+1}\in Ap_{n+1}\Leftrightarrow v_{n} \in  (A+C)p_{n+1},
			\label{eq:proof7}
		\end{align}
  where 
  \begin{equation*}
      v_{n} = S(p_{n+1}-y_{n}) + (u_n -u_{n+1})/\gamma_{n}-(Cy_{n}-Cp_{n+1}).
  \end{equation*}
It follows from \eqref{eq:proofc1} that, $v_{n}\rightarrow 0$. Therefore, since $C$ is cocoercive, the operator $A+C$ 
 is maximally monotone \cite[Corollary~25.5]{bauschkebook2017} and we deduce from the weak-strong closeness of its graph \cite[Proposition~20.38]{bauschkebook2017} and \eqref{eq:proof7} that every weak cluster point of $\left(x_{n}\right)_{n\in \mathbb{N}}$ belongs to $\bm{Z}$. The result follows by applying Opial’s lemma (see \cite[Lemma 2.47]{bauschkebook2017}).
  \end{proof}
 \begin{rem}	\label{rem:partcases2} In the case when $\lambda = 1$, $\rho_{n}$, $\eta_{n}$, and  $\xi_{n}$ defined in \eqref{eq:defrho}, \eqref{eq:defetan}, and \eqref{eq:defxin}, reduce to $
 	    	\rho_{n}=1 -\frac{\gamma_{n}}{2\mu}-\zeta_{n},$  
 	    	$\eta_{n} = (1-\alpha_{n})\rho_{n}-\zeta_{n-1}$, and $
 	    	\xi_{n} = \alpha_{n}\left(1+\alpha_{n}\right)+\alpha_{n}(1-\alpha_{n})\rho_{n}$. 	Therefore, in this case, the  conditions in \eqref{eq:desalphabeta2} correspond to, for every $n \geq N_0 \in \N$,
 	\begin{equation}\label{eq:rempartcases12}
    \begin{cases}
        1-\frac{\gamma_{n}}{2\mu}-\zeta_{n}\geq 0,\\
        (1-\alpha_{n})\rho_{n}-\zeta_{n-1}-\alpha_{n+1}((1+\alpha_{n+1})+(1-\alpha_{n+1})\rho_{n+1})\geq \epsilon.
    \end{cases}
 	\end{equation}
 	In the non-inertial case ($\alpha_n  \equiv 0$), \eqref{eq:rempartcases12} reduces to, for every $n \geq N_0$, $1-\gamma_{n}/(2\mu) -\zeta_{n}-\zeta_{n-1}\geq \epsilon,$
 	which is the condition in \cite[Theorem~3.1]{MorinBanertGiselsson2022} for guaranteeing the convergence of NFBM.
\end{rem}	
\subsection{NFBM with double inertia and additional momentum}
In this subsection, we study the convergence of Algorithm~\ref{algo:NFBDoubleInertial}.
The following proposition is a preliminary step before addressing the convergence.
\begin{prop}\label{prop:descnDI}
  In the context of Problem~\ref{pro:main} and Assumption~\ref{asume:1}, consider the sequence $(x_{n})_{n \in \N}$ defined recursively by Algorithm~\ref{algo:NFBDoubleInertial} with initialization points $(x_0,x_{-1}, u_{0})\in \H^3$. Let $x \in \bm{Z}$ and, for every $n \in \N$, define
  \begin{align}
        \tilde{\alpha}_{n} &= \alpha_{n}+\theta_{n},\label{eq:defTalpha}\\
      \eta_{n} &= \left(1-\tilde{\alpha}_{n}-\frac{\gamma_{n}(1-\beta_{n})}{2\mu} -\zeta_{n}(1-\alpha_{n})-\zeta_{n-1}\right),\label{eq:defetaDI}\\
     \xi_{n} &= \left(2\tilde{\alpha}_{n}-\frac{\gamma_{n}\beta_{n}(1-\beta_{n})}{2\mu} -\zeta_{n}\alpha_{n}(1-\alpha_{n}) \right),\label{eq:defxiDI}\\ 
     C_{n+1}(x)&= \|x_{n+1}-x\|^2_S-\tilde{\alpha}_{n}\|x_{n}-x\|^2_S+2\scal{u_{n+1}}{x_{n+1}-x} \nonumber \\\label{eq:defcnDI}
     &\hspace{2cm}+\zeta_{n}\|x_{n+1}-y_{n}\|^2_S+\xi_{n+1}\|x_{n+1}-x_{n}\|^2.
\end{align}
Suppose that there exists $N_0 \in \N$  such that $(\tilde{\alpha}_{n})_{n \geq N_0}$ is non-decreasing and that
\begin{equation}\label{eq:desalphabetaDI}
    (\forall n \geq N_0) \quad \left(\tilde{\alpha}_{n}-\frac{\gamma_{n}\beta_{n}}{2\mu}-\zeta_{n}\alpha_{n} \right)\geq 0.
\end{equation}
Then, the following hold.
\begin{enumerate}
\item\label{prop:descn0DI} For every $n\geq N_0$, $\xi_{n}$ is a non-negative.
    \item\label{prop:descn1DI} For every $n \in \N$,
    \begin{equation}\label{eq:descn1DI} 
    C_{n+1}(x)\leq  C_{n}(x)
	-(\eta_{n}-\xi_{n+1})\|x_{n+1}-x_{n}\|^{2}.
    \end{equation}
  \end{enumerate}  
  Moreover, suppose that there exists  $\epsilon \in \RPP$ such that 
\begin{equation}\label{eq:desetaxi2}
	 (\forall n \geq N_0) \quad \eta_{n}-\xi_{n+1}\geq \epsilon.
	\end{equation}
    \begin{enumerate}[resume]
\item\label{prop:descn2DI} $(C_{n}(x))_{n \geq N_0}$  is a non-negative convergent sequence.
\item\label{prop:descn3DI}  $\sum_{n \in \N} \|x_{n+1}-x_{n}\|^2 < +\infty$.
\end{enumerate}
\end{prop}
 \begin{proof}
\begin{enumerate}
    \item It follows directly from \eqref{eq:desalphabetaDI} by noticing that
    \begin{equation*}
        (\forall n \geq N_0) \quad \xi_{n} \geq \left(\tilde{\alpha}_{n}-\frac{\gamma_{n}\beta_{n}}{2\mu}-\zeta_{n}\alpha_{n} \right).
    \end{equation*}
    \item   Fix $n \geq N_0$ and define $\tilde{y}_{n}=y_{n}+\theta_{n}(x_{n}-x_{n-1})=x_{n}+\tilde{\alpha}_{n}(x_{n}-x_{n-1})$. It follows from \eqref{eq:algoNFBDoubleInertial} that
    	\begin{align*}
		\gamma_{n} M_{n}y_{n}&-\gamma_{n} Cz_{n}+u_{n}+\theta_{n} S(x_{n}-x_{n-1})-\gamma_{n} M_{n} x_{n+1}\in \gamma_{n} A x_{n+1}\nonumber\\
  &\Leftrightarrow 
  S\tilde{y}_{n} + u_{n}  - (Sx_{n+1}+u_{n+1})-\gamma_{n} Cz_{n} \in \gamma_{n} A x_{n+1}.
	\end{align*}
  Let $x \in \bm{Z}$, thus, $-\gamma_{n} Cx \in \gamma_{n} A x$. Therefore, by the monotonicity of $A$, we have 
  \begin{equation}\label{eq:monADI}
         \scal{S\tilde{y}_{n} + u_{n}  - (Sx_{n+1}+u_{n+1})-\gamma_{n} Cz_{n}+\gamma_n Cx}{x_{n+1}-x}\geq 0.
\end{equation}
Let us bound the terms in \eqref{eq:monADI}. First, note that, \eqref{eq:prodinternonorma} and \eqref{eq:normaalpha} yield
\begin{align}\label{eq:proofp1DI}
2\scal{\tilde{y}_{n} -  x_{n+1}}{x_{n+1}-x}_S
    =& \|\tilde{y}_{n}-x\|_S^2-\|\tilde{y}_{n}-x_{n+1}\|_S^2-\|x_{n+1}-x\|_S^2\nonumber\\
	=& (1+\tilde{\alpha}_{n})\| x_{n}-x\|_S^{2}+\tilde{\alpha}_{n}(1+\tilde{\alpha}_{n}) \|x_{n}-x_{n-1}\|_S^2\nonumber\\ 
  \quad \quad &-\tilde{\alpha}_{n}\|x_{n-1}-x\|^{2} - \|x_{n+1}-\tilde{y}_{n}\|_S^{2}\nonumber\\ 
  &\quad -\|x_{n+1}-x\|_S^{2}.
\end{align}
Now, since $\gamma_{n}M_{n}-S$ is $\zeta_{n}$-Lipschitz, by \eqref{algo:NFBDoubleInertial} we have that
\begin{align}\label{eq:proofp2DI}
       2&\scal{u_{n} -u_{n+1}}{x_{n+1}-x}\nonumber\\
        &=  2\scal{u_n}{x_{n}-x}-2\scal{ u_{n+1}}{x_{n+1}-x}+2\scal{u_{n}}{x_{n+1}-x_{n}}\nonumber\\
        &\leq  2\scal{u_n}{x_{n}-x}-2\scal{u_{n+1}}{x_{n+1}-x}+2(\|u_n\|_{S^{-1}}\|{x_{n+1}-x_{n}}\|_S)\nonumber\\
        &\leq  2\scal{u_n}{x_{n}-x}-2\scal{ u_{n+1}}{x_{n+1}-x}\nonumber\\
        &\quad +\zeta_{n-1}(\|x_{n}-y_{n-1}\|^2_S+\|{x_{n+1}-x_{n}}\|^2_S)
\end{align}
Moreover, by the $\mu$-cocoercivity of $C$  we have
	\begin{align}\label{eq:proofp3DI}
	2\gamma_{n}&\scal{Cz_{n}-Cx}{x-x_{n+1}}\nonumber\\
    &= 
	2\gamma_{n}\scal{Cz_{n}-Cx}{x-z_{n}}+2\gamma_{n}\scal{Cz_{n}-Cx}{z_{n}-x_{n+1}}\nonumber\\
	&\leq
	-2\gamma_{n}\mu\|Cz_{n}-Cx\|^{2}_{S^{-1}}+2\gamma_{n}\mu\|Cz_{n}-Cx\|^{2}_{S^{-1}}
	+\frac{\gamma_{n}}{2\mu}\|x_{n+1}-z_{n}\|^{2}_S\nonumber\\
	&=\frac{\gamma_{n}}{2\mu}\|x_{n+1}-z_{n}\|^{2}_S.
\end{align}
Now, from \eqref{eq:monADI}-\eqref{eq:proofp3DI} and the non-decreasing property of $(\tilde{\alpha}_{n})_{n \geq N_0}$ we have
\begin{align}\label{eq:proofp4DI}
    \|&x_{n+1}-x\|^2_S-\tilde{\alpha}_{n}\|x_{n}-x\|^2_S+2\scal{u_{n+1}}{x_{n+1}-x} +\zeta_{n}\|x_{n+1}-y_{n}\|^2_S\nonumber\\
    &\leq\|x_{n}-x\|^2_S-\tilde{\alpha}_{n-1}\|x_{n-1}-x\|^2_S+2\scal{u_n}{x_{n}-x}+\zeta_{n-1}\|x_{n}-y_{n-1}\|^2_S\nonumber\\
    &\quad  +\tilde{\alpha}_{n}(1+\tilde{\alpha}_{n}) \|x_{n}-x_{n-1}\|^2_S - \|x_{n+1}-\tilde{y}_{n}\|^2_S+\frac{\gamma_{n}}{2\mu}\|x_{n+1}-z_{n}\|^2_S\nonumber\\ 
    &\quad +\zeta_{n}\|x_{n+1}-y_{n}\|^2_S+ \zeta_{n-1}\|x_{n+1}-x_{n}\|^2_S.
\end{align}
Note that, by \eqref{eq:normaalpha} we have, for every $\sigma \in \R$, that
\begin{align*}
&\|x_{n+1}-(x_{n}+\sigma (x_{n}-x_{n-1}))\|^{2}\\
&=(1-\sigma)\|x_{n+1}-x_{n}\|^{2}-\sigma(1-\sigma)\|x_{n}-x_{n-1}\|^{2}+\sigma\| x_{n+1}-2x_{n}+x_{n-1}\|^{2}
\end{align*}
Therefore,
\begin{align}\label{eq:proofp5DI}
     - \|x_{n+1}&-\tilde{y}_{n}\|^2_S+\frac{\gamma_{n}}{2\mu}\|x_{n+1}-z_{n}\|^2_S+\zeta_{n}\|x_{n+1}-y_{n}\|^2_S\nonumber\\ 
     &= -\left(1-\tilde{\alpha}_n-\frac{\gamma_n(1-\beta_n)}{2\mu}-\zeta_n(1-\alpha_n)\right)\|x_{n+1}-x_{n}\|^{2}\nonumber\\ 
     &\quad + \left((1-\tilde{\alpha}_n)\tilde{\alpha}_n-\frac{\gamma_n(1-\beta_n)\beta_n}{2\mu}-\zeta_n(1-\alpha_n)\alpha_n\right)\|x_{n}-x_{n-1}\|^{2}\nonumber\\
     &\quad-\left(\tilde{\alpha}_{n}-\frac{\gamma_{n}\beta_{n}}{2\mu}-\zeta_{n}\alpha_{n} \right)\|x_{n+1}-2x_{n}+x_{n-1}\|^{2}.
\end{align}
Then, by replacing \eqref{eq:proofp4DI} in \eqref{eq:proofp5DI} and in view of\eqref{eq:desalphabetaDI} we obatin \eqref{eq:descn1DI}.
Finally, since $\eta_{n}\geq \xi_{n+1}\geq0$, we conclude that $-\tilde{\alpha}_{n-1}\geq-\tilde{\alpha}_{n}\geq-(1-\zeta_{n-1})$. Then, the proof of \ref{prop:descn2DI} and \ref {prop:descn3DI} are analogous to the proof of Proposition~\ref{prop:descn}.
\end{enumerate}
   \end{proof}
  The following result establishes the convergence of Algorithm~\ref{algo:NFBDoubleInertial} and its proof is analogous to the proof of Theorem~\ref{teo:main}. 
\begin{teo}\label{teo:mainDI}
  In the context of Problem~\ref{pro:main} and Assumption~\ref{asume:1}, consider the sequence $(x_{n})_{n \in \N}$ defined recursively by Algorithm~\ref{algo:NFBDoubleInertial} with initialization points $(x_0,x_{-1},u_{0})\in \H^3$. Let $(\xi_{n})_{n \in \N}$, $(\eta_{n})_{n \in \N}$, and $(\tilde{\alpha}_{n})_{n \in \N}$ be the sequences defined in \eqref{eq:defxiDI}, \eqref{eq:defetaDI}, and \eqref{eq:defTalpha}, respectively.
Suppose that there exists $N_0 \in \N$  such that $(\tilde{\alpha}_{n})_{n \geq N_0}$ is non-decreasing and that
\begin{equation}\label{eq:desalphabeta2DI}
   (\forall n \geq N_0) \quad \quad \left(\tilde{\alpha}_{n}-\frac{\gamma_{n}\beta_{n}}{2\mu}-\zeta_{n}\alpha_{n} \right)\geq 0 \text{ and } \eta_{n}-\xi_{n+1}\geq \epsilon.
\end{equation}
Then, $(x_{n})_{n \in \N}$ converges weakly to a point in $\bm{Z}$.
\end{teo}
 \begin{proof} Let $x\in \bm{Z}$. In view of \eqref{eq:algoNFBDoubleInertial}, we have 
\begin{align}\label{eq:proofc1D1}
	&M_{n}y_{n}-M_nx_{n+1}+u_{n}/\gamma_{n}+\theta_{n} S(x_{n}-x_{n-1})/\gamma_{n}-Cz_{n}\in Ax_{n+1}\nonumber\\
    &\Leftrightarrow S(\tilde{y}_{n}-x_{n+1}) + (u_n -u_{n+1})/\gamma_{n}-(Cz_{n}-Cx_{n+1}) \in  (A+C)x_{n+1}.
\end{align}
By Proposition~\ref{prop:descnDI}.\ref{prop:descn3DI}, we have $\sum_{n \in \N} \|x_{n+1}-x_{n}\|^2<+\infty$. Then, since, for every $n \in \N$, $\tilde{y}_{n} = x_{n} + \tilde{\alpha}_{n} (x_{n} - x_{n-1})$, $z_{n} = x_{n} + \beta_{n} (x_{n} - x_{n-1})$, $\gamma_{n} M_{n}-S$ is $(1-\varepsilon)$-Lipschitz, and $C$ is $(1/\mu)$-Lipschitz, we conclude that
\begin{equation}\label{eq:proofc1D2}
    \|\tilde{y}_{n}-x_{n+1}\|\to 0, \quad  \|x_{n+1}-y_{n}\|\to 0, \quad
\|u_{n}\|\to 0, \text{ and } \|Cx_{n+1}-Cz_{n}\|\to 0.   
    \end{equation}
Therefore, by \eqref{eq:proofc1D1} and \eqref{eq:proofc1D2}, we deduce that every weak cluster point of $\left(x_{n}\right)_{n\in \mathbb{N}}$ belongs to $\bm{Z}$. Moreover, by Proposition~\ref{prop:descnDI}.\ref{prop:descn2DI}  we conclude that $(C_{n}(x))_{n \in \N}$ is non-negative and convergent. Consequently, by proceeding similarly to the proof of Theorem~\ref{teo:main}, we deduce that $(\|x_{n}-x\|)_{n \in \N}$ is convergent and the weak convergence follows from Opial’s lemma. \end{proof}
\begin{rem}\label{rem:partcasesDI}
	 In the case when, for every $n \in \N$, $\alpha_{n}=\beta_{n}=0$, \eqref{eq:algoNFBDoubleInertial} reduces to
    \begin{equation*}
		(\forall n \in \N) \left\lfloor
\begin{aligned}
			&x_{n+1} = (M_{n}+A)^{-1} \left(M_{n}x_{n}-Cx_{n}+u_{n}/\gamma_{n}+\theta_{n}S(x_{n}-x_{n-1})/\gamma_n\right),\\
   &u_{n+1} = (\gamma_{n} M_{n} - S) x_{n+1}- (\gamma_{n} M_{n} - S) x_{n},\end{aligned}
  \right.
	\end{equation*} 
 which is the momentum version of NFBM proposed in \cite[Algorithm~2]{MorinBanertGiselsson2022}. In this case,  \eqref{eq:desalphabeta2DI} reduces to,
     \begin{equation*}
		(\forall n \geq N_0)\quad  1-\theta_{n}-2\theta_{n+1}-\frac{\gamma_{n}}{2\mu}-\zeta_{n}-\zeta_{n-1} \geq \varepsilon,
	\end{equation*} 
 which corresponds with the condition in \cite[Corollary~4.1]{MorinBanertGiselsson2022}.
\end{rem}
\section{Particular Cases of the Inertial NFBM}\label{sec:PC}
In this section, we describe particular instances of Algorithm~\ref{algo:NFBInertial} and Algorithm~\ref{algo:NFBDoubleInertial} obtaining inertial versions of existing methods in the literature. This section is developed in the context of Problem~\ref{pro:main} and Assumption~\ref{asume:1}.
\subsection{Forward-Backward} In the particular case when $S=\id$ and, for all $n \in \N$, $\gamma_n =\gamma\in \RPP$ and  $\gamma M_{n} = \id$, Algorithm~\ref{algo:NFBInertial} can be written as
\begin{equation*}
		(\forall n \in \N)\quad \left\lfloor
\begin{aligned}
			&y_{n} = x_{n}+\alpha_{n} (x_{n}-x_{n-1}),\\
			&p_{n+1} = J_{\gamma A} \left(y_{n}-\gamma Cy_{n} \right),\\
			&x_{n+1} = (1-\lambda) y_{n} + \lambda p_{n+1}, 
		\end{aligned}
  \right.
	\end{equation*} 
which corresponds to FB with inertia and relaxation step.  In this case, since $\zeta_{n} \equiv 0$ and assuming $\alpha_{n}\nearrow \alpha$, in view of Theorem~\ref{teo:main}, the convergence of $(x_{n})_{n \in \N}$ to a $x \in \zer (A+C)$ is guaranteed if
\begin{equation*}
	\frac{(1-\alpha)^2}{\lambda}\left(2-\lambda-\frac{\gamma}{2\mu}\right)-\alpha(1+\alpha) >0
\end{equation*}
which corresponds with the condition in \cite[Corollary~3.12]{AttouchCabot19}. Furthermore, if $\lambda = 1$, it reduces to
\begin{equation}\label{eq:desaFB}
	1-3\alpha-\frac{\gamma(1-\alpha)^2}{2\mu} >0,
\end{equation}
which is the condition proposed in \cite{LorenzInFB2015}.
 In this same setting  Algorithm~\ref{algo:NFBDoubleInertial} iterates
\begin{equation*}
	(\forall n \in \N)\quad \left\lfloor
	\begin{aligned}
		&y_{n} = x_{n}+\alpha_n (x_{n}-x_{n-1}),\\
		&z_{n} = x_{n}+\beta_n (x_{n}-x_{n-1}),\\
		&x_{n+1} = J_{\gamma A} \left(y_{n}-\gamma Cz_{n} \right).
	\end{aligned}
	\right.
\end{equation*} 
Moreover, if we assume that $\alpha_{n}\nearrow\alpha$ and $\beta_{n} \to\beta$, in view of \eqref{eq:desalphabeta2DI}, the convergence of this algorithm is guaranteed if
\begin{equation}\label{eq:desabFB}
	\alpha - \frac{\gamma\beta}{2\mu} \geq  0 \textnormal{ and }  1-3\alpha-\frac{\gamma}{2\mu}(1-\beta)^2 > 0.
\end{equation}
Note that \eqref{eq:desabFB} restricts $\alpha$ to the interval $]0,1/3[$, just as \eqref{eq:desaFB} does; however, $\beta$ is not constrained to this range, thus, \eqref{eq:desabFB} allows more flexibility in the choice of $(\alpha,\beta)$ than \eqref{eq:desaFB}. 
\subsection{Forward-Half-Reflected-Backward} Let $B\colon \H \to \H$ be a $\zeta$-Lipschitz operator for $\zeta \in \RPP$ and $\tilde{A}\colon \H \to 2^\H$ be a set-valued operator such that $\tilde{A}+B$ is maximally monotone. The problem is to, 
\begin{equation} \label{eq:probFHRB}
		\text{find }  x \in \H  \text{ such that } 
		0 \in (\tilde{A}+B+C)x.
	\end{equation}
We assume that the solution set is not empty. This problem can be solved, for example, by the methods proposed in \cite{BricenoDavis2018,Malitsky2020SIAMJO}. Set $A=\tilde{A}+B$, $S=\id$, and, for every $n \in \N$, $\gamma_{n}M_{n}=\id-\gamma B$, for $\gamma \in \RPP$. In this setting, \eqref{eq:algoNFBInertial} is written as follows. 
\begin{equation}\label{eq:algoIMT}
		(\forall n \in \N)\quad \left\lfloor
\begin{aligned}
			&y_{n} = x_{n}+\alpha_n (x_{n}-x_{n-1}),\\
			&p_{n+1} = J_{\gamma \tilde{A}} \left(y_{n}-\gamma (Bx_{n} + Cy_{n})-\gamma(By_{n}-By_{n-1})\right),\\
   &x_{n+1} = (1-\lambda)y_{n}+\lambda p_{n+1}.
		\end{aligned}
  \right.
	\end{equation} 
Note that, for every $n \in \N$, $\gamma_{n}M_{n}-S$ is $(\zeta\gamma)$-Lipschitz. Then, if $\alpha_{n}\nearrow \alpha \in \RPP$, according to Theorem~\ref{teo:main}, if
\begin{equation}\label{eq:condIMT1}
(1-\alpha)^2\left(2-\lambda -(1+2|1-\lambda|)\zeta\gamma -\frac{\gamma}{2\mu}\right)-\lambda^2\zeta\gamma -\lambda\alpha(1+\alpha)>0,
\end{equation}
$(x_{n})_{n \in \N}$, generated by \eqref{eq:algoIMT}, converges weakly to some solution to \eqref{eq:probFHRB}. The recurrence in \eqref{eq:algoIMT} differs from the algorithm proposed in \cite[Theorem~4.3]{Malitsky2020SIAMJO} which is limited to $\lambda < 1$. 
\begin{rem} In the case when $\lambda = 1$, by defining $\kappa = \zeta+1/2\mu$ and $\varphi \colon \alpha \to (1-3\alpha)/(\zeta+\kappa(1-\alpha)^2)$, 
\eqref{eq:condIMT1} reduces to
\begin{equation*}
1-\zeta\gamma-\widetilde{\gamma}-(3-2\widetilde{\gamma})\alpha-\widetilde{\gamma}\alpha^2>0\Leftrightarrow 
\varphi(\alpha)>\gamma.
\end{equation*}
By noticing that $\varphi$ is decreasing in $[0,1]$, we conclude that while $\alpha$ increases, $\gamma$ decreases, and conversely.
\end{rem}
In the same setting,  Algorithm~\ref{algo:NFBDoubleInertial} iterates as follows.
\begin{equation}\label{eq:algoDIMT}
	(\forall n \in \N)\left\lfloor
	\begin{aligned}
		&y_{n} = x_{n}+\alpha_n (x_{n}-x_{n-1}),\\
		&z_{n} = x_{n}+\beta_n (x_{n}-x_{n-1}),\\
		&x_{n+1} = J_{\gamma \tilde{A}} \left(y_{n}-\gamma (Bx_{n} + Cz_{n}+By_{n}-By_{n-1})+\theta_n (x_{n}-x_{n-1})\right).
	\end{aligned}
	\right.
\end{equation} 
In view of Theorem~\ref{teo:mainDI},  if  $\alpha_{n}\nearrow \alpha \in \RPP$, $\beta_{n} \to \beta \in \RPP$, $\theta_{n} \nearrow \theta \in \RPP$, $(x_{n})_{n \in \N}$, generated by \eqref{eq:algoDIMT}, converges weakly to a solution to \eqref{eq:probFHRB} if
\begin{equation}\label{eq:condDIMT1}
	1-3(\alpha+\theta)-\frac{\gamma(1-\beta)^2}{2\mu}-\gamma\zeta-\gamma\zeta(1-\alpha)^2 >0 \textnormal{ and } \alpha+\theta-\frac{\gamma\beta}{2\mu}-\zeta\gamma\alpha>0.
\end{equation}
Note that, if $\alpha=\beta=0$, \eqref{eq:condDIMT1} reduces to $1-3\theta-\gamma/(2\mu) -2\zeta\gamma > 0$
which corresponds with the condition proposed in \cite[Theorem~4.3]{Malitsky2020SIAMJO} and in \cite[Theorem~3.4]{Tang2022InFRB}. Note that the inertial versions of FRB and FHRB proposed in \cite{Malitsky2020SIAMJO,Tang2022InFRB} considers only the momentum term $\theta(x_{n}-x_{n-1})$ on its iterations. On the other hand, the algorithm in \eqref{eq:algoIMT} includes two inertial steps which are evaluated in the terms $Cz_{n}$ and $y_{n} -\gamma_{n-1}(By_{n}-By_{n-1})$, giving more flexibility in the method implementation.
\subsection{Primal-Dual with Block-Triangular Resolvent}\label{sec:PDBT} Let $\G$ be a real Hilbert space,  let $A_{1}\colon \H \to 2^\H$ and $A_2 \colon \G \to 2^\G$ be maximally monotone operators, let $B\colon \H \to \H$ be a monotone and $\zeta$-Lipschitz operator for $\zeta \in \RPP$, let $\tilde{C}$ be a $\mu$-cocoercive operator for $\mu \in \RPP$, and let $L \colon \H \to \G$ be a linear bounded operator. In this case, the problem is to
		\begin{equation}\label{eq:probPD}
			\text{find }  (x,u) \in \H \times \G \text{ such that } 
			\begin{cases}
				0 &\in (A_1+B+\tilde{C})x+L^*u\\
				0 &\in A_2^{-1}u-Lx.
			\end{cases} 
		\end{equation}
		under the hypothesis that its solution set is not empty.
This problem and particular instances of it, have been studied, for example, in \cite{BricenoDavis2018,ChambollePock2011,Condat13,MorinBanertGiselsson2022,Roldan4ops,Vu13}. Let $(\sigma,\tau)\in \RPP^2$, set $\HH=\H\times \G$, and consider the following operators.
\begin{equation*}
\begin{aligned}
    &A \colon \HH \to 2^\HH \colon (x,u) \mapsto ((A_1+B)x+L^*u)\times (A_2^{-1} u -Lx),\\
    &C \colon \HH \to \HH \colon (x,u) \to (\tilde{C}x,0),\\
    &S \colon \HH \to \HH \colon (x,u) \to (x-\tau L^*u,\tau u/\sigma-\tau Lx),\\
    &M \colon \HH \to \HH \colon (x,u) \to (x/\tau-Bx-L^*u,u/\sigma+Lx).
\end{aligned}
\end{equation*}
Additionally, for every $n \in \N$, set $\alpha_{n}=\alpha \in \RPP$ and define $\gamma_{n}=\tau$ and $M_{n} = M$. Hence, in this setting, by proceeding similar to \cite[Section~6.1]{MorinBanertGiselsson2022},  for initialization points $y_{-1}\in\H,$ $(x_0,v_0) \in \HH,$ and  $(x_{-1},v_{-1}) \in \HH$, Algorithm~\ref{algo:NFBInertial} iterates as follows
\begin{equation}
	\label{eq:algoPD1Inertial}
	(\forall n \in \N)\quad \left\lfloor
	\begin{aligned}
		&(y_{n},w_n) = (x_{n},v_n)+\alpha (x_{n}-x_{n-1},v_n-v_{n-1}),\\
		&p_{n+1} = J_{\tau A_1} \left(y_{n}-\tau L^*w_{n} -\tau(Bx_{n}+By_{n}-By_{n-1}+\tilde{C}y_{n})\right),\\
		&q_{n+1} = J_{\sigma A_2^{-1}}(w_{n}+\sigma L(2x_{n+1}-y_{n})),\\
		&(x_{n+1},v_{n+1})=(1-\lambda)(x_{n},v_{n})+\lambda (p_{n+1},q_{n+1}),		
		\end{aligned}
	\right.
\end{equation} 
which corresponds to a relaxed and inertial version of the 
Primal-Dual with Block-Triangular Resolvent (PDBTR) algorithm proposed in \cite[Section~6.1
]{MorinBanertGiselsson2022}.
According to \cite[Proof of Corollary~6.1]{MorinBanertGiselsson2022}, if $\kappa:=1-\sigma\tau \|L\|^2 > 0$, $A$ is maximally monotone, $S$ is linear self-adjoint and strongly monotone, $C$ is $(\mu\kappa)^{-1}$-cocoercive with respect to $S$,  and $\tau M- S$ is $(\tau\zeta/\kappa)$-Lipschitz with respect to $S$.  Therefore, the condition in \eqref{eq:desalphabeta2} reduces to
\begin{equation}\label{eq:condDIPD1}
	(1-\alpha)^2\left(2-\lambda -\frac{(1+2|1-\lambda|)\tau\zeta}{\kappa} -\frac{\tau}{2\mu\kappa}\right)-\frac{\lambda^2\tau\zeta}{\kappa} -\lambda\alpha(1+\alpha)>0.
\end{equation}
Then, if \eqref{eq:condDIPD1} holds, $(x_{n},v_{n})_{n \in \N}$ converges weakly to a solution to \eqref{eq:probPD}.  When $B=0$, \eqref{eq:algoPD1Inertial} reduces to the inertial and relaxed version of CV and of CP if $C=0$ \cite{ChambollePock2011,Condat13,MaulenFierroPeypouquet2023,Vu13}. In the case when $\alpha =0$ and $\lambda = 1$, \eqref{eq:condDIPD1} reduces to the condition in \cite[Corollory~6.1]{MorinBanertGiselsson2022} guaranteeing the convergence of PDBTR.
 \section{Numerical Experiments}\label{se:NE}
  In this section, we present a series of numerical experiments\footnote{All numerical experiments were implemented in MATLAB on a laptop equipped with an AMD Ryzen 5 3550Hz processor, Radeon Vega Mobile Gfx, and 32 GB of RAM. The code is available at this \href{https://github.com/cristianvega1995/Relaxed-and-Inertial-Nonlinear-Forward-Backward-with-Momentum}{repository}.} on image restoration to evaluate the efficiency of the inertial variants of FHRB. In particular, we test the performance of our method under different inertia and momentum parameter settings. Additionally, we compare the results with the standard FHRB method (without inertia) and its momentum-based variant proposed in \cite{Malitsky2020SIAMJO,Tang2022InFRB}. The problem formulation is described first, followed by the presentation and discussion of the numerical results.
 \subsection{Numerical experiments on image restoration}
Let $n \in \N$ and $x^{*} \in \mathcal{C}:=[0,255]^{N\times N}$ represent an image of $N\times N$ pixels in the range $[0,1]$. The goal is to recover the original image from a blurry and noisy observation $b = Kx^{*} + \epsilon$,
where $K \colon \mathbb{R}^{N \times N} \to \mathbb{R}^{N \times N}$ is a linear, bounded operator modelling a blur process and $\epsilon$ is a random additive noise. We assume that $x^{*}$ can be well-approximated by solving the following optimization problem.
\begin{align}\label{Min:image} \min_{x\in \mathcal{C}} & \frac{1}{2}\|Kx - b\|^{2} + \rho \|Dx\|_{1}, \end{align}
where $\rho > 0$ is a regularization parameter, $\|\cdot\|_{1}$ denotes the $\ell_{1}$ norm, and $D\colon x \mapsto (D_{1}x, D_{2}x)$ is the discrete gradient with $D_1$ and $D_2$ representing the horizontal and vertical differences with Neumann boundary conditions, respectively. By setting $f = \iota_{\mathcal{C}}$, $ h = \frac{1}{2}\|K\cdot - b\|^{2}$, and  $g = \lambda \|\cdot\|_{1}$, we can note that $0 \in \sri(\dom g - D (\dom f))$, thus, in view of \cite[Theorem 16.3 \& Theorem~16.47]{bauschkebook2017}, the optimization problem \eqref{Min:image} is equivalent to
\begin{equation}\label{eq:conopt}
    \text{find } x^{*} \in \H \text{ such that } 
    0 \in \partial f(x^{*}) + D^{*}\partial g(Dx^{*})   + \nabla h(x^{*}),
\end{equation}
which, together with its dual problem, is a particular instance of \eqref{eq:probFHRB} and can be solved using the algorithms \eqref{eq:algoIMT} and \eqref{eq:algoDIMT}.
 Indeed, by taking $u^* \in \partial g (Dx^*)$, from \eqref{eq:conopt}, we have $(0,0)^\top \in (\tilde{A}+B+C)(x^*,u^*)^\top$
 where $\tilde{A}\colon (x,u)\to \partial f (x) \times \partial g^*(u)$, $B \colon (x,u) \to (D^*u,-Dx)$, and $C\colon (x,u) \to (\nabla h (x), 0)$.
Since $f \in \Gamma_0(\H)$ and $g \in \Gamma_0(\G)$ 
by \cite[Proposition~20.22 \& Proposition~20.23]{bauschkebook2017}, $A$ is maximally monotone. Moreover, since $B$ is skew, by \cite[Proposition 2.7]{Briceno2011Skew}, $B$ is $\|D\|$-Lipschitz and $\tilde{A}+B$ is maximally monotone.  Moreover, by \cite[Theorem 5.1]{Tang2022InFRB}, $C$ is $(1/|K\|^{2})$-cococercive. Therefore, we can apply \eqref{eq:algoIMT} and \eqref{eq:algoDIMT} to solve the problem in \eqref{Min:image}. In this context, Algorithm~\ref{algo:NFBInertial} and Algorithm~\ref{algo:NFBDoubleInertial} are written in \eqref{eq:algoDIMT-functions} in a general framework. We compare the algorithms proposed in Table~\ref{Tab:IP} and Table~\ref{Tab:RP} that are particular instances of the recurrence in \eqref{eq:algoDIMT-functions}.
\begin{equation}\label{eq:algoDIMT-functions}
	(\forall n \in \N)\left\lfloor
	\begin{aligned}
		&(y_{n}^{1},y_n^2) = (x_{n}^{1},x_{n}^{2})+\alpha (x_{n}^{1}-x_{n-1}^{1},x_{n}^{2}-x_{n-1}^{2}),\\
		&z_{n}^{1} = x_{n}^{1}+\beta(x_{n}^{1}-x_{n-1}^{1}),\\
            &(w_{n}^{1},w_n^2) = (x_{n}^{1},x_{n}^{2})+\theta (x_{n}^{1}-x_{n-1}^{1},x_{n}^{2}-x_{n-1}^{2}),\\
		&p_{n+1}^{1} = \prox_{\gamma f} \left(w^{1}_{n}-\gamma (D^{*}(x_{n}^{2}+y_n^2-y_{n-1}^2) + \nabla h(z_{n}^{1}))\right),\\
		&p_{n+1}^{2} = \prox_{\gamma g^{*}} \left(w_{n}^{2}-\gamma D(x_{n}^{1}+y_{n}^{1}-y_{n-1}^{1})\right),\\ &(x_{n+1}^{1},x_{n+1}^{2}) = (1-\lambda)(y_{n}^{1},y_{n}^{2})+\lambda (p_{n+1}^{1}, p_{n+1}^{2}).
	\end{aligned}
	\right.
\end{equation} 
The explicit formula of $\prox_{\gamma f}$ can be found in  \cite[Example 23.4 \& Proposition 29.3]{bauschkebook2017}. While the explicit formula of $\prox_{\gamma g^*}$ can be found in  \cite[Proposition 24.8 (ix) \& Example 24.11]{bauschkebook2017}.
\subsection{Numerical results}
In a first instance, we consider $N=256$ and the original image $x^{*}$ shown in Figure~\ref{fig:x256}.  The operator $K$ corresponds to an average blur kernel of $3\times 3$ with symmetric boundary conditions, implemented in MATLAB using the \texttt{imfilter} function.  We approximate $\zeta=\|D\|$ as $\sqrt{8}$ (see \cite{Chambolle2004}). Additionally, we have $\|K\|=1$, thus, $\mu=1/\|K\|^{2}=1$. To test the inertial parameters, for a given step-size $\gamma$, we select $\alpha$, $\beta$, and $\theta$ to satisfy either \eqref{eq:condIMT1} or \eqref{eq:condDIMT1}, depending on the chosen algorithm. In particular, we consider
$\gamma = 2\mu\kappa/(1+4\mu\zeta)$ for $\kappa\in \left\{0.5, 0.6, 0.7, 0.8\right\}$. For FHRBSI, FHRBI, and FHRBDI, we analyze three distinct cases of $\alpha$ and $\theta$. The specific values of $\alpha$, $\beta$, and $\theta$ are summarized in Table~\ref{Tab:IP}, where we set $\widetilde{\gamma} =\gamma(\zeta+1/(2\mu))$.
{\begin{center}
	\begin{table}\centering\resizebox{12cm}{!} {
	\begin{tabular}{|c|c|c|c|c|c|c|}\cline{1-6}
		Algorithm & Case & $\alpha$ & $\beta$ & $\theta$ &$\lambda$ & \multicolumn{1}{}{} \\ 
        \hline
		 FHRB& 1& 0& 0& 0& 1 & - \\	\hline  
  \multirow{3}{*}{ FHRBSI}& 1& 0&  0& $\theta_1/3$& 1 & \multirow{3}{*}{$\theta_1=\dfrac{0.99}{3}\left(1-\widetilde{\gamma}-\zeta\gamma\right)$}\\	\cline{2-6}
 &2 &0&  0& $2\theta_1/3$& 1 &\\	\cline{2-6}
  &3 & 0&  0& $\theta_1$& 1 &\\	\hline 
    \multirow{3}{*}{ FHRBI}& 1&$\alpha_{1}/3$& 0& 0& 1 &\multirow{3}{*}{$\alpha_1=\dfrac{0.99}{2\widetilde{\gamma}}(2\widetilde{\gamma}-3+\sqrt{(3-2\widetilde{\gamma})^2+4(1-\zeta\gamma-\widetilde{\gamma})\widetilde{\gamma}})$}\\	\cline{2-6}
    & 2 &$2\alpha_{1}/3$&  0& 0& 1 &\\	\cline{2-6}
    &3 &$\alpha_{1}$&  0& 0&1 &\\	\hline 
      \multirow{3}{*}{ FHRBDI}& 1&$\alpha_{2}/3$& 1& 0& 1 & \multirow{3}{*}{$\alpha_2=\dfrac{0.99}{2\zeta\gamma}\left(2\zeta\gamma-3+\sqrt{\left(3-2\zeta\gamma\right)^{2}+4\zeta\gamma\left(1-2\zeta\gamma-\frac{(1-\beta)^{2}\gamma}{2\mu}\right)}\right)$}\\	\cline{2-6}
      &2&$2\alpha_{2}/3$& 1& 0& 1 &\\	\cline{2-6}
      &3&$\alpha_{2}$& 1& 0& 1 &\\
      \hline\
		 FHRBSDI& 1 & 0& 1& $\theta_2$& 1 & $\theta_2 = 0.99\left( 1-\gamma(1-\beta)^2/(2\mu)-2\zeta\gamma\right)/3$ \\
      \hline
	\end{tabular}}
\captionsetup{width=\textwidth}     \caption{Inertial parameters for FRHB,  FHRB semi inertial (FHRBSI) \cite[Theorem 5.1]{Tang2022InFRB}, FHRB inertial (FHRBI), FHRB double inertial (FHRBDI), and FHRB semi double inertial (FHRBSDI).}\label{Tab:IP}
	\end{table}
\end{center}}
{\begin{center}
	\begin{table}\centering\resizebox{12cm}{!} {
	\begin{tabular}{|c|c|c|c|c|c|c|}\cline{1-6}
		Algorithm & Case & $\alpha$ & $\beta$ & $\theta$ &$\lambda$ & \multicolumn{1}{}{} \\ 
        \hline
        \multirow{4}{*}{ FHRBRI}& 1&$3\alpha_{1}/4$& 0& 0& $\lambda_1$ & \multirow{4}{*}{$\lambda_1=\dfrac{0.99}{2\zeta\gamma}\left(\sqrt{\left(1+2\zeta\gamma+\alpha(1+\alpha)\right)^{2}+4\zeta\gamma\left(2+\zeta\gamma-\frac{\gamma}{2\mu}\right)}-1-2\zeta\gamma-\alpha(1+\alpha)\right)$}\\	\cline{2-6}
      &2&$\alpha_{1}/2$& 0& 0& $\lambda_1$ &\\	\cline{2-6}
      &3&$\alpha_{1}/4$& 0& 0& $\lambda_1$ &\\	\cline{2-6}
      &4&$0$& 0& 0& $\lambda_1$ &\\
      \hline
	\end{tabular}}
    \captionsetup{width=\textwidth} \caption{Relaxation and inertial parameters for FHRBRI.}\label{Tab:RP}
	\end{table}
\end{center}}
We compare the aforementioned algorithms, step-sizes, and relaxation parameters across 20 random realizations of $b$. The stopping criterion is based on the relative error, with a tolerance of $10^{-6}$, and a maximum of $10^4$ iterations. Table~\ref{Tab:1} reports the average iteration number (IN) and the average CPU Time (T) in seconds, obtained by applying all the algorithms to
solve the optimization problem in \eqref{Min:image} for the 20 random observations. From the results in Table~\ref{Tab:1}, we observe that larger inertial parameters improve the convergence of the inertial algorithms. Among these, FHRBDI achieves the best performance in terms of the average number of iterations. However, FHRBSDI exhibits the best performance in terms of average CPU time, even though it requires more iterations. This can be attributed to the additional inertial step performed at each iteration by FHRBDI, which increases its computational cost. Furthermore, we note that larger values of $\kappa$, corresponding to larger step-sizes $\gamma$, lead to better results in terms of both iterations and CPU time for all algorithms. This observation aligns with findings reported in \cite{Tang2022InFRB,BotrelaxFBF2023}. 

To test FHRB relaxed inertial (FHRBRI), we consider $\kappa=0.8$ and use the inertial and relaxation parameters specified in Table~\ref{Tab:RP}, where $\alpha_1$ is defined in Table~\ref{Tab:IP}. The results are summarized in Table~\ref{Tab:2}. From these results, we observe that FHRBRI achieves improved convergence compared to FHRB; however, it is outperformed by FHRBDI and FHRBSDI in terms of iteration number and CPU time, respectively. We conclude that to further accelerate the convergence of FHRBRI, prioritizing larger values of $\alpha$ over larger values of $\lambda$ is recommended. 

As mentioned earlier, larger values of $\gamma$ yield better results in terms of iterations and CPU time for all the algorithms. However, as $\kappa$ approaches 1, the parameters $\alpha$, $\beta$, and $\theta$ satisfying \eqref{eq:condDIMT1} converge to 0, causing the inertial effect to diminish. To explore this behavior, we set $\kappa=0.99$ and consider two scenarios for the inertial parameters. First, we test the algorithms using constant step-sizes that violates \eqref{eq:condDIMT1}. The results, presented in Table~\ref{Tab:3}, show that FRHR outperforms all the results reported in Table~\ref{Tab:1}. On the other hand, we observe that including an inertial parameter can accelerate the convergence but may also lead to divergence. In particular, FHRBSI, FHRBI, and FHRBDI exhibit accelerated convergence when $\alpha =0.1$ but diverge when $\alpha =0.25$. Notably, FHRBI and FHRBDI also accelerate for $\alpha =0.2$. Figure~\ref{fig:noteo} illustrates the relative error as a function of the iteration number for the 10th random observation. 

To leverage the acceleration provided by the inertial step even when $\kappa=0.99$, while ensuring convergence, we propose a {\it restart} strategy for the inertial parameter. In particular, we define $\alpha_n = \alpha \in \RPP$ for $n\leq N_0\in \N$ and $\alpha_n=0$ for $n\geq N_0$. This choice of $(\alpha_n)_{n \in \N}$ satisfies the assumptions of Theorem~\ref{teo:main}. Table~\ref{Tab:4} presents the convergence results for FHRB, FHRBSI, FHRBI, FHRBDI, and the restart strategy, called FHRBIR. The inertial parameters for FHRBSI, FHRBI, and FHRBDI were selected based on Table~\ref{Tab:IP}, specifically in their respective case 3. As these parameters are relatively small, Table~\ref{Tab:4} shows that the acceleration effect is negligible in these cases. On the other hand, the restart strategy improves convergence when $\alpha=0.1$ and $\alpha =0.2$ but slows down when $\alpha =0.25$. Figure~\ref{fig:restart1} illustrates the relative error as a function of the iteration number for the 10th random observation. Furthermore, Figure~\ref{fig:recoveredimages1} displays the original, blurred and noisy, and recovered images for this observation.

To conclude this section, we modify the scenario by comparing FHRB with FHRBIR in cases where $K$ is generated using a blur of kernel of size $9\times 9$ for $N=256$ and $N=512$. Additionally, we evaluate FHRBIR with restart points at 1000, 2000, and 3000 iterations. The results are summarized in Table~\ref{Tab:5}.  From these results, we observe that the acceleration effect  is more pronounced in this scenario due to the increased computational cost per iteration. Among the restart strategies tested, the one with a restart at 3000 iterations achieves the best performance. The relative error as a function of the iteration number for the 10th random observation is shown in Figure~\ref{fig:restart2} and Figure~\ref{fig:restart3}. Additionally, the original, blurred and noisy, and recovered images for this observation are presented in Figure~\ref{fig:recoveredimages2} and Figure~\ref{fig:recoveredimages3}.
 {\begin{center}
	\begin{table}\centering\resizebox{12cm}{!}{
	\begin{tabular}{|c|c|c|c|c|c|c|c|c|c|}
		\cline{3-10}
\multicolumn{2}{}{ }& \multicolumn{2}{|c|}{$\kappa = 0.5$}& \multicolumn{2}{|c|}{$\kappa = 0.6$}& \multicolumn{2}{|c|}{$\kappa = 0.7$}& \multicolumn{2}{|c|}{$\kappa = 0.8$}\\ \hline
		Algorithm & Case &IN &T& IN  & T & IN  &T&IN  &T\\ 
        \hline
		\hline
		 FHRB&  1& 1762&10.73&1588&9.51&1451&8.75&1343&8.12\\	\hline \hline 
  \multirow{3}{*}{ FHRBSI}&1&1704&10.71&1545&9.65&1421&8.88&1326&8.25\\	\cline{2-10}
  &2&1647&10.30&1501&9.33&1392&8.63&1309&8.12\\	\cline{2-10}
  &3& 1588&9.87&1457&9.08&1362&8.46&1293&7.99\\	\hline \hline
    \multirow{3}{*}{ FHRBI}&1& 1694&10.70&1535&9.67&1412&8.85&1320&8.37\\	\cline{2-10}
    &2& 1624&10.26&1480&9.32&1374&8.71&1296&8.17\\	\cline{2-10}    &3&1555&\bf{9.82}&1424&8.94&1336&8.44&1273&8.03\\	\hline \hline
      \multirow{3}{*}{ FHRBDI}&1&1563&10.26&1430&9.28&1337&8.72&1269&8.32\\	\cline{2-10}
      &2& 1558&10.25&1425&9.26&1333&8.67&1266&8.29\\	\cline{2-10}
      &3& \bf{1546}&10.03&\bf{1411}&9.19&\bf{1321}&8.61&\bf{1256}&8.18\\
      \hline\hline
		 FHRBSDI& 1& 1573&9.87&1441&\bf{8.88}&1346&\bf{8.37}&1276&\bf{7.89}\\
        \hline
	\end{tabular}}\caption{Numerical results for FHRB, FHRBSI, FHRBDI, FHRBDI, and FHRBSDI for $\kappa \in \{0.5, 0.6, 0.7, 0.8\}$}\label{Tab:1}
	\end{table}
\end{center}}
	\begin{table}
\begin{minipage}[t]{4.5cm}
	\resizebox{4.5cm}{!}{\begin{tabular}{|c|c|c|c|}
	\cline{3-4}
	\multicolumn{2}{}{ }& \multicolumn{2}{|c|}{$\kappa = 0.8$}\\ \hline
	Algorithm & Case & IN &T\\ 
	\hline
	\hline
	\multirow{4}{*}{ FHRBRI}&1&1277&8.38\\	\cline{2-4}
	&2&1281&8.40\\	\cline{2-4}
	&3& 1286&8.38\\	\cline{2-4}
	&4& 1291&7.96\\	\hline 
	\end{tabular}}
\captionsetup{width=\textwidth} \caption{Numerical results for FHRBRI for $\kappa =0.8$.} \label{Tab:2}
\end{minipage}\hspace{0.5cm}
\begin{minipage}[c]{7cm}
		\resizebox{7cm}{!}{\begin{tabular}{|c|c|c|c|c|c|c|}
			\cline{6-7}
			\multicolumn{5}{}{ }& \multicolumn{2}{|c|}{$\kappa = 0.99$}\\ \hline
			Algorithm & $\alpha$ & $\beta$ & $\theta$ & $\lambda$ & IN &T\\ 
			\hline
			\hline
			 FHRB&  0& 0&0&1&1194&7.32\\	\hline \hline 
			\multirow{3}{*}{ FHRBSI}&$0$& 0&0.1&1& 1123& 7.22\\	\cline{2-7}&$0$& 0&0.2&1& --& --\\	\cline{2-7}
			&$0$& 0&0.25&1& --& --\\	\hline \hline
			\multirow{3}{*}{ FHRBI}&$0.1$& $0.1$&0&1& 1124& 7.26\\	\cline{2-7}
			&$0.2$& $0.2$&0&1&1051&6.73\\	\cline{2-7}
			&$0.25$& $0.25$&0&1&--&--\\	\hline \hline
			\multirow{3}{*}{ FHRBDI}&$0.1$& 1&0&1&1123&7.48\\	\cline{2-7}
			&$0.2$& 1&0&1&1050&6.98\\	\cline{2-7}
			&$0.25$& 1&0&1&--&--\\
			\hline
	\end{tabular}}
\captionsetup{width=\textwidth} \caption{Numerical results for FHRB, FHRBSI, FHRBI, FHRBDI, for $\kappa=0.99$ and   inertial parameters that do not satisfy the hypothesis guaranteeing convergence.}\label{Tab:3}
\end{minipage}
	\end{table}
{\begin{center}
	\begin{table}\centering\resizebox{10cm}{!} {
	\begin{tabular}{|c|c|c|c|c|c|c|c|}
		\cline{7-8}
\multicolumn{6}{}{ }& \multicolumn{2}{|c|}{$\kappa = 0.99$}\\ \hline
		Algorithm & $\alpha$ & $\beta$ & $\theta$ & $\lambda$ & $N_0$ & IN &T\\ 
        \hline
		\hline
		 FHRB&  0& 0&0&1&-&1194&7.16\\	\hline \hline 
  FHRBSI&$0$& 0& $\theta_1$ &1& -& 1192& 7.37\\	\hline \hline 
    FHRBI & $\alpha_1$& 0 &0&1&-& 1190& 7.46\\\hline \hline
      FHRBDI&$\alpha_2$& $0.05$ &0&1&-&1188&7.54\\
        \hline\hline
    \multirow{3}{*}{ FHRBIR}&$0.1$& $0.1$&0&1&1000& 1085  & 6.78  \\	\cline{2-8}
    &$0.2$& $0.2$&0&1&1000& {\bf 998} & {\bf 6.26} \\	\cline{2-8}
    &$0.25$& $0.25$&0&1&1000&1689&10.34\\	\hline
	\end{tabular}}\caption{Numerical results for FHRB, FHRBSI, FHRBI, FHRBDI, and FHRBIR, for $\kappa=0.99$.}\label{Tab:4}
	\end{table}
\end{center}}
{
\begin{center}
	\begin{table}\centering \resizebox{10cm}{!}{
	\begin{tabular}{|c|c|c|c|c|c|c|c|c|c|}
		\cline{7-10}
\multicolumn{6}{}{ }& \multicolumn{4}{|c|}{$\kappa = 0.99$}\\ 
\cline{7-10}
\multicolumn{6}{}{ }& \multicolumn{2}{|c|}{$N=256$}& \multicolumn{2}{|c|}{$N = 512$}\\\hline
		Algorithm & $\alpha$ & $\beta$ & $\theta$ & $\lambda$ &$N_0$& IN &T & IN &T\\ 
        \hline
		\hline
		 FHRB&  0& 0&0&1& - &3593&24.13 &  4477&188.10\\	\hline \hline 
    \multirow{3}{*}{ FHRBIR}&$0.2$& $0.2$&0&1&1000& 3349  & 22.82 & 4230 & 178.86\\	\cline{2-10}
    &$0.2$& $0.2$&0&1&2000& 3108 & 21.47 & 3983 & 169.40 \\	\cline{2-10}
    &$0.2$& $0.2$&0&1&3000&{\bf 3015} & {\bf 21.13} & \bf{3738} & \bf{160.23}\\	\hline
	\end{tabular}}\caption{ Numerical results for FHRB and FHRBIR, for $\kappa=0.99$, blur of kernel $9\times 9$, and $N=256$ and $N=512$,}\label{Tab:5}
	\end{table}
\end{center}}

\begin{figure}
\centering
\subfloat[\scriptsize $\alpha=0.1$]{\label{fig:noteo01}\includegraphics[scale=0.26]{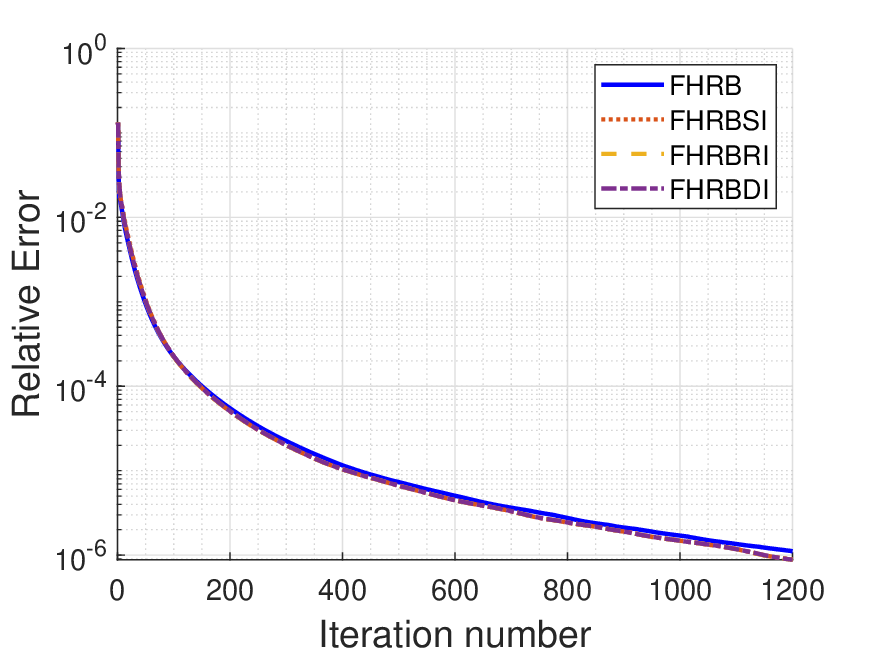}}\,
\subfloat[\scriptsize $\alpha=0.2$]{\label{fig:noteo02}\includegraphics[scale=0.26]{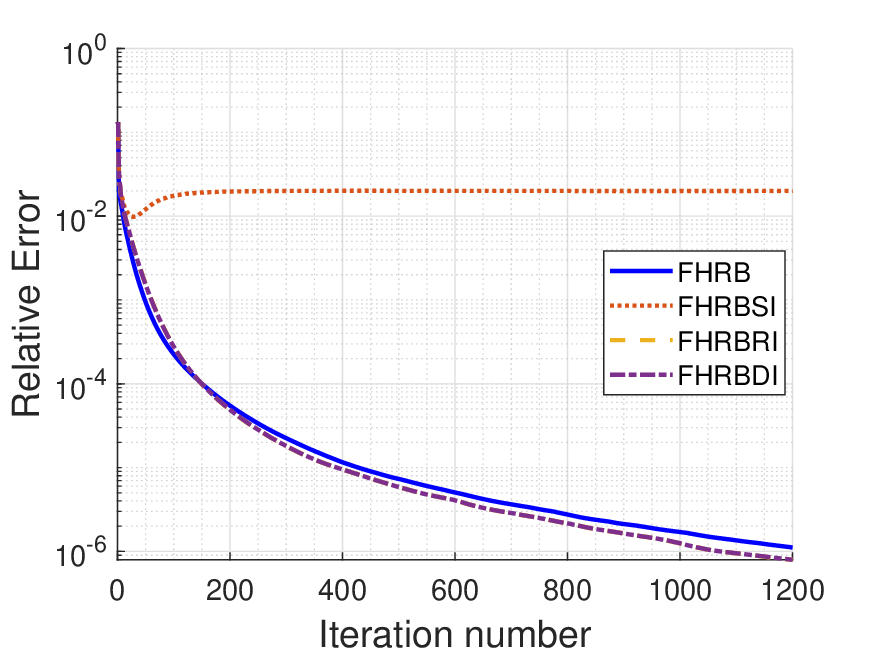}}
\subfloat[\scriptsize $\alpha=0.25$]{\label{fig:noteo03}\includegraphics[scale=0.26]{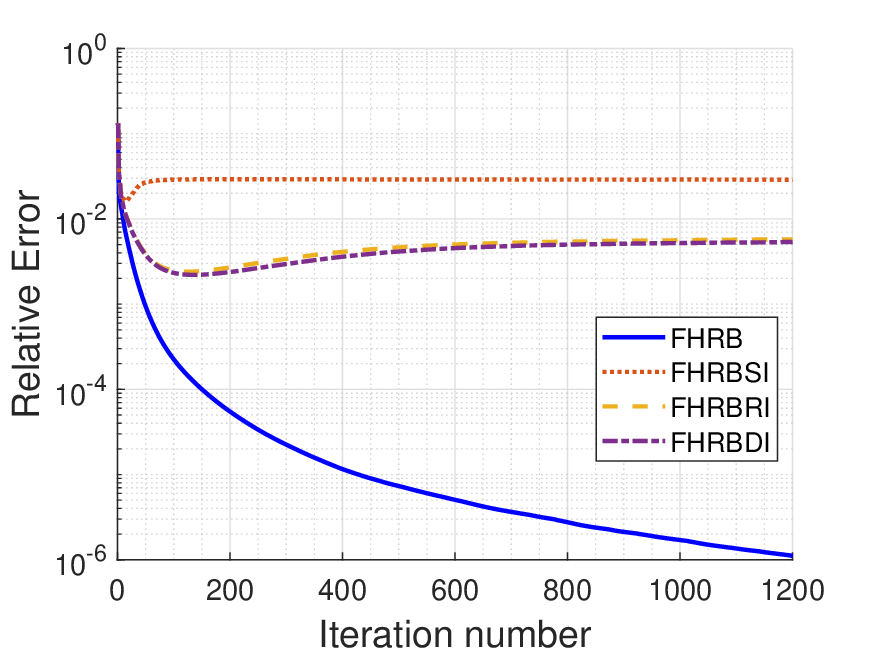}}
\captionsetup{width=\textwidth} \caption{Relative error along iteration number for the random observation 10. In this case, the inertial parameters do not satisfy the hypothesis guaranteeing convergence. See Table~\ref{Tab:3} for details on the parameters.} 
\label{fig:noteo}
\end{figure}

\begin{figure}
\centering
\subfloat[\scriptsize Blur $3\times 3$, $N=256$]{\label{fig:restart1}\includegraphics[scale=0.26]{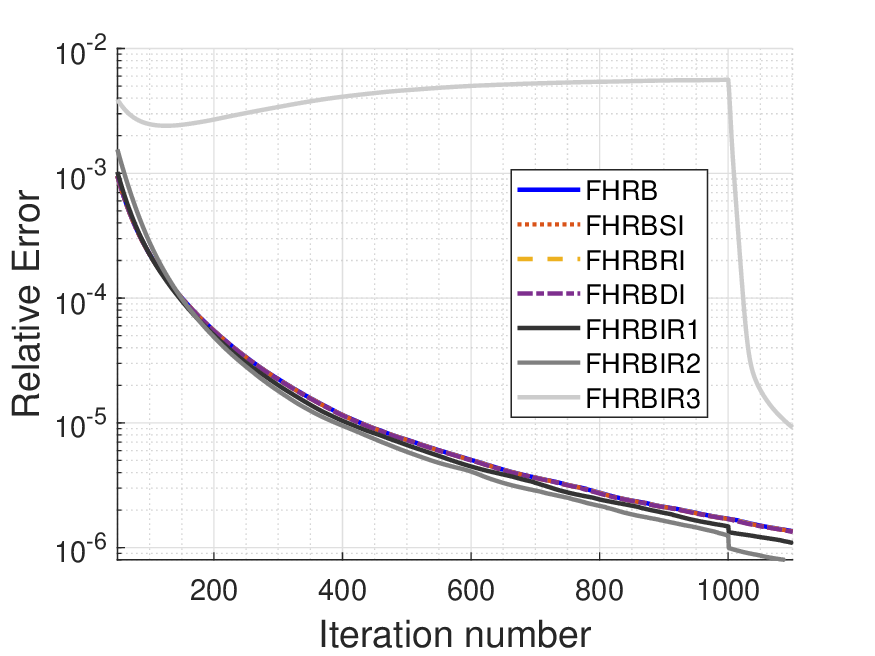}}\,
\subfloat[\scriptsize Blur $9\times 9$, $N=256$ ]{\label{fig:restart2}\includegraphics[scale=0.26]{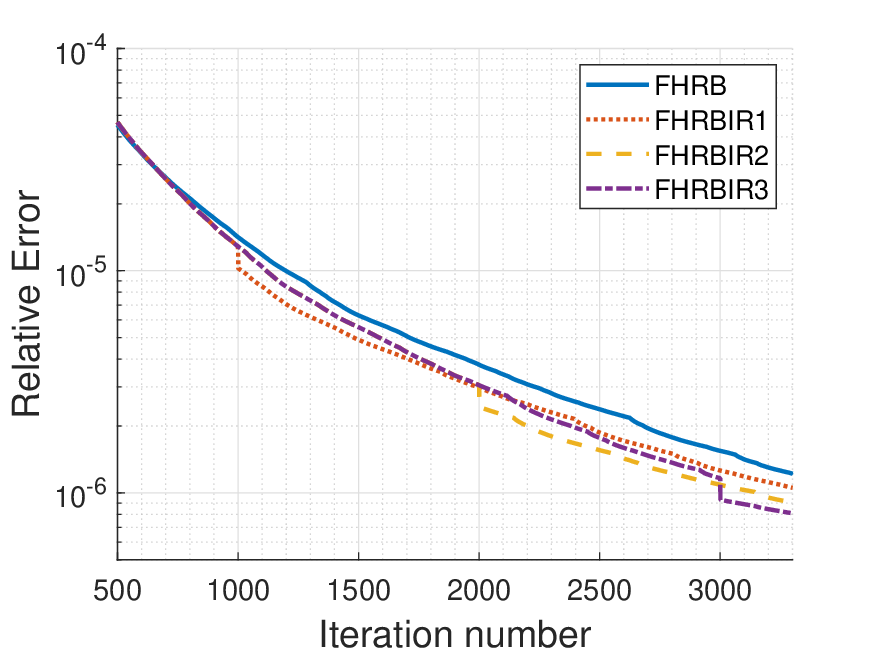}}
\subfloat[\scriptsize Blur $9\times 9$,$N=512$]{\label{fig:restart3}\includegraphics[scale=0.26]{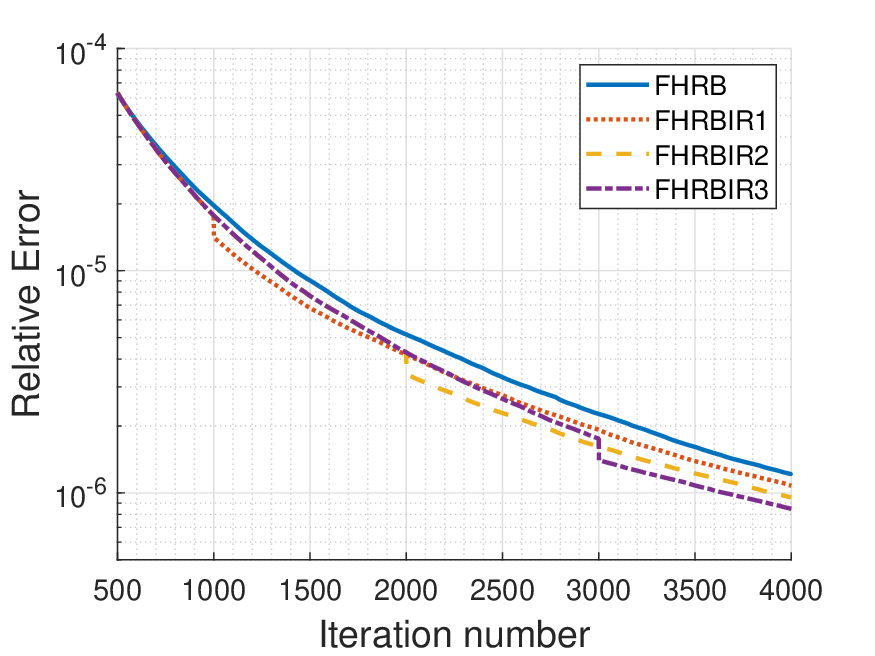}}
\captionsetup{width=\textwidth} \caption{Relative error along iteration number for the random observation 10. See Table~\ref{Tab:4} for details on the parameters.} 
\label{fig:restart}
\end{figure}

\begin{figure}
\centering
\subfloat[\scriptsize $x^*$]{\label{fig:x256}\includegraphics[scale=0.3]{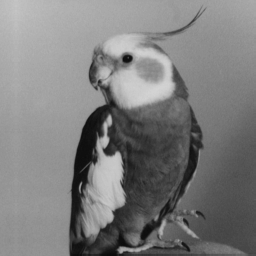}}\,
\subfloat[\scriptsize $b_{10}$ (27.54)]{\label{fig:xb256}\includegraphics[scale=0.3]{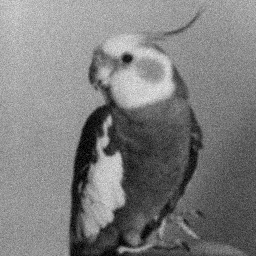}}\,
\subfloat[\scriptsize FHRB (33.42)]{\label{fig:xFRHF256}\includegraphics[scale=0.3]{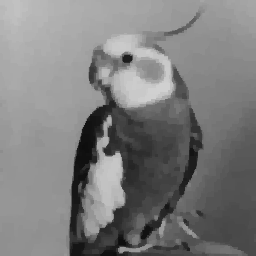}}\,
\subfloat[\scriptsize FHRBSI (33.42)]{\label{fig:xFRHFS256}\includegraphics[scale=0.3]{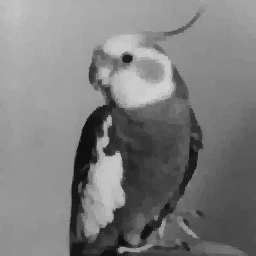}}\\
\subfloat[\scriptsize FHRBRI (33.42)]{\label{fig:xFRHFI256}\includegraphics[scale=0.3]{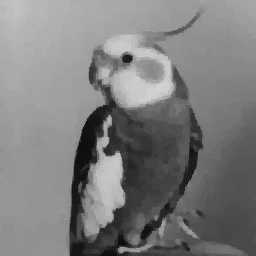}}\,
\subfloat[\scriptsize FHRBDI (33.42)  ]{\label{fig:xFRHFDI256}\includegraphics[scale=0.3]{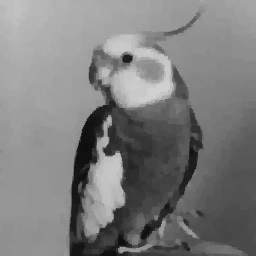}}\,
\subfloat[\scriptsize FHRBIR (33.42)]{\label{fig:xFRHFR256}\includegraphics[scale=0.3]{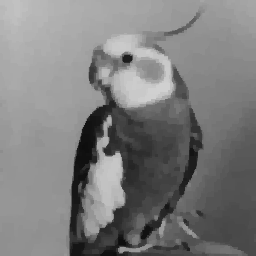}}
\captionsetup{width=\textwidth} \caption{Original image, blur and noisy observation 10, and recovered images for  FHRB,  FHRBSI, FHRBI, FHRBDI, and FHRBIR ($\alpha = 0.2$) with their respective PNSR (dB), blur of kernel $3\times 3$ and $N=256$. See Table~\ref{Tab:4} for details on the parameters.} 
\label{fig:recoveredimages1}
\end{figure}

\begin{figure}
\centering
\subfloat[\scriptsize $x^*$]{\label{fig:x256b9}\includegraphics[scale=0.3]{x_original.png}}\,
\subfloat[\scriptsize $b_10$ (24.84)]{\label{fig:xb256b9}\includegraphics[scale=0.3]{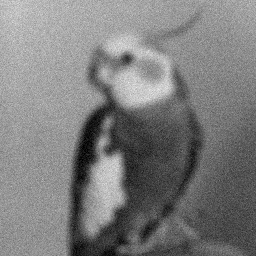}}\,
\subfloat[\scriptsize FHRB (28.71)]{\label{fig:xFRHF256b9}\includegraphics[scale=0.3]{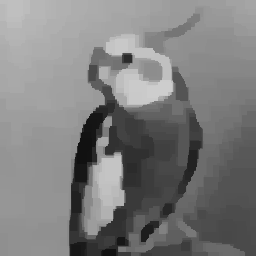}}\,
\subfloat[\scriptsize FHRBR (28.70)]{\label{fig:xFRHFR256b9}\includegraphics[scale=0.3]{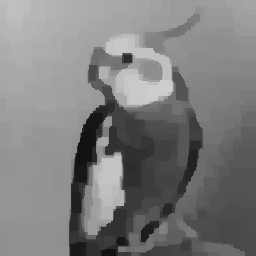}}
\captionsetup{width=\textwidth} \caption{Original image, blur and noisy observation 10, and recovered images for  FHRB and FHRBIR ($N_0=3000$) with their respective PNSR (dB), blur of kernel $9\times 9$ and $N=256$. See Table~\ref{Tab:5} for details on the parameters.} 
\label{fig:recoveredimages2}
\end{figure}

\begin{figure}
\centering
\subfloat[\scriptsize $x^*$]{\label{fig:x512}\includegraphics[scale=0.155]{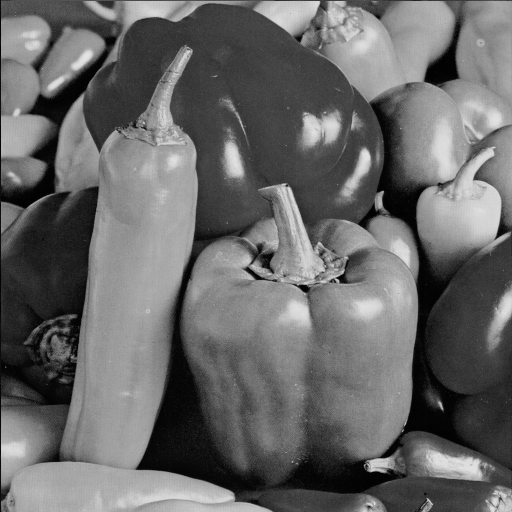}}\,
\subfloat[\scriptsize $b_10$ (23.68)]{\label{fig:xb512b9}\includegraphics[scale=0.155]{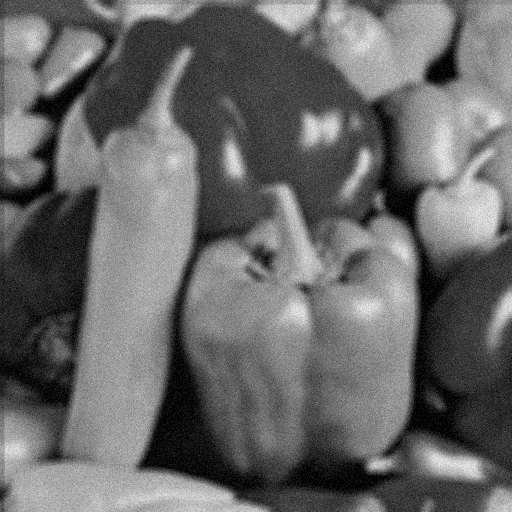}}\,
\subfloat[\scriptsize FHRB (26.41)]{\label{fig:xFRHF512b9}\includegraphics[scale=0.155]{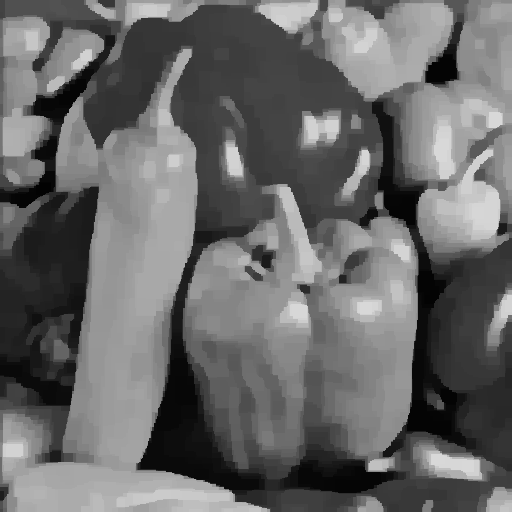}}\,
\subfloat[\scriptsize FHRBR (26.41)]{\label{fig:xFRHFR512b9}\includegraphics[scale=0.155]{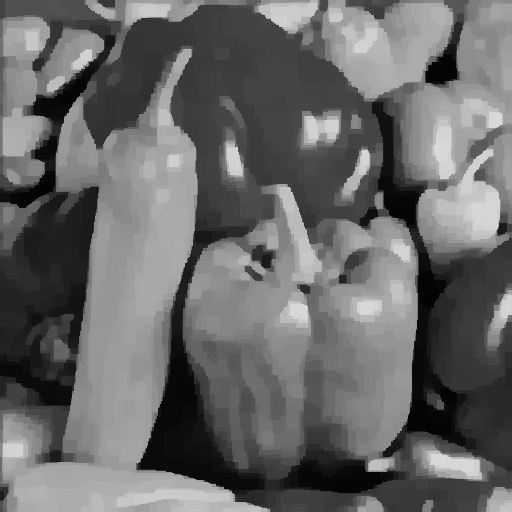}}
\captionsetup{width=\textwidth} \caption{Original image, blur and noisy observation 10, and recovered images for  FHRB and FHRBIR ($N_0=3000$) with their respective PNSR (dB), blur of kernel $9\times 9$ and $N=512$. See Table~\ref{Tab:5} for details on the parameters.} 
\label{fig:recoveredimages3}
\end{figure}
 \section{Conclusions}\label{se:conc}
In this article, we propose several inertial/relaxed versions of NFBM, extending and recovering the classic convergence result for NFBM. Moreover,  by a specific choice of monotone operators and metrics in the inertial/relaxed version of NFBM, we recover and extend inertial and relaxed versions of FB, FHRB, CP, CV, among others. We compare the FHRB with its momentum, inertial, relaxed, and double-inertial versions in image restoration. For a fixed step-size, all the inertial/relaxed versions improve the convergence, with the double-inertial version exhibiting the best performance in terms of the number of iterations. However, since the inertial parameters converge to zero as the step-size approaches its admissible limit, the acceleration becomes negligible. To leverage the acceleration provided by the inertial step, we propose a restart strategy for the inertial parameter. Numerical experiments illustrate that this strategy improves convergence, although there is no theoretical framework to implement it to accelerate the process. These results motivate further investigation into incorporating larger step-sizes that do not limit the inertial parameter.\\
{\small 
{\bf Conflict of interest}: The authors declare that they have no conflicts of interest that are relevant to the content of this article.}

\end{document}